\documentclass[11pt,reqno]{amsart}
\usepackage{amsmath,mathtools}
\usepackage{graphicx}
\usepackage{epsfig,epsf,psfrag}
\usepackage{amsfonts}
\usepackage{setspace}
\usepackage{url}
\usepackage{color}
\usepackage[font=small]{caption}
\usepackage[font=footnotesize]{subcaption}
\allowdisplaybreaks

\usepackage[top=1in, bottom=1in, left=1.2in, right=1.2in]{geometry}
\usepackage{afterpage}
\usepackage{bm}
\usepackage{multicol}
\usepackage{multirow} 
\usepackage{float}

\usepackage{isomath}
\usepackage{longtable}
\usepackage[section]{placeins}
\usepackage[colorlinks,bookmarksopen,bookmarksnumbered,citecolor=blue, linkcolor=blue, urlcolor=blue]{hyperref}

\usepackage{comment}
\excludecomment{mycomment}

\usepackage{booktabs}
\usepackage{array}
\usepackage{siunitx}

\newtheorem{theorem}{Theorem}[section]

\newtheorem{lemma}[theorem]{Lemma}

\newtheorem{remark}[theorem]{Remark}

\title[]{A Spectral Solver for Acoustic  Scattering by Multiple Quasi-Axisymmetric Structures}

\author{Jun Lai}
\address{School of Mathematical Sciences, Zhejiang University,
	Hangzhou, Zhejiang 310027, China, and Center for Interdisciplinary Applied Mathematics, Zhejiang University, Hangzhou, Zhejiang 310027, China}
\email{laijun6@zju.edu.cn}
\author{Yuxin Li}
\address{School of Mathematical Sciences, Zhejiang University,
	Hangzhou, Zhejiang 310027, China}
\email{yuxinli@zju.edu.cn}

\date{}

\begin{document}
\begin{abstract}
Acoustic scattering arises in a wide range of applications, including medical imaging, geophysical exploration, acoustic metamaterials, etc. In this paper, we develop a fast and highly accurate algorithm for acoustic scattering by multiple quasi-axisymmetric objects, whose axis of rotation is an arbitrary curve. The method is based on a Nystr{\"o}m discretization that combines Gauss-Legendre quadrature with the trapezoidal rule. To treat the singular integrals that occur when target points are close to or coincide with source points, we reformulate them as evaluations of the modal Green's function and its derivatives, which are computed efficiently using the fast Fourier transform and convolution. The multiple scattering solver is then constructed by coupling the single scatterer discretizations through inter-body boundary integral interactions. We also present a convergence analysis for scattering problems with smooth geometries. Numerical examples demonstrate the efficiency and accuracy of the proposed method for solving multiple scattering problems involving up to 1000 quasi-axisymmetric structures.
\end{abstract}
\maketitle 
\section{Introduction}

Acoustic scattering plays a key role in many important applications, including medical imaging, geophysical exploration, non-destructive testing, and the design of acoustic metamaterials~\cite{article44,article45}. In many of these settings, waves interact with multiple scatterers simultaneously, and accurate prediction of the resulting collective scattering effects is essential for simulations and optimizations \cite{article56,article57}. However, the computation of multiple scattering suffers from slow convergence and low-order accuracy, especially in three dimensions. The goal of this paper is to develop  a high-order solver for acoustic scattering by multiple quasi-axisymmetric structures. 

When considering acoustic scattering problems in homogeneous media, boundary integral equation (BIE) methods  are particularly attractive because they reduce the dimensionality by one and enforce the radiation condition automatically through the Green's function representation \cite{article41,article13}. Hence, compared with volumetric methods such as finite element and finite difference methods, BIE avoids truncating the unbounded domain and typically requires far fewer unknowns to achieve high accuracy. At the same time, the quality of a boundary integral solver depends critically on how the boundary operators are discretized. For acoustic objects with smooth surfaces, common discretization strategies include panel-based collocation, Galerkin schemes, and Nystr\"om methods based on corrected quadratures or high-order product integration \cite{article49, article48,article13}.  Among these approaches,  Nystr\"om discretizations are especially appealing because they combine implementation simplicity with high-order accuracy, provided that singular and nearly singular interactions are treated carefully.

For a single obstacle, the difficulty arises when the source and target points are close. In the multiple scattering setting~\cite{article57}, in addition to the self-interaction of each obstacle, one must account for a large number of inter-body interactions, and the resulting global linear system couples all scatterers simultaneously. As the number of bodies increases, the costs of matrix assembly, storage, and iterative solution can grow rapidly. Existing methods for multi-particle scattering in a variety of settings, including layered media, electromagnetic scattering, and elastic scattering, demonstrate both the importance and the difficulty of this regime \cite{article52,article55,article53,article54}. An effective solver must balance two distinct requirements: accurate treatment of singular or nearly singular self-interactions, and efficient handling of the many smooth inter-body interactions. Much previous work has focused on scattering by axisymmetric objects because  their geometry allows substantial analytical and numerical simplification \cite{article14, article59, article19,article20}.  The key advantage of the axisymmetric setting is rotational invariance: the surface integral equation on a two-dimensional surface can be decomposed into a sequence of one-dimensional integral equations posed on the generating curve \cite{article46,article47}. This reduction is enabled by modal Green's functions, which expand the three-dimensional Helmholtz kernel into azimuthal Fourier modes and thereby decouple the problem mode by mode \cite{article59}. Building on this idea, high-order discretizations for axisymmetric scattering have been developed using panel-based Nystr\"om schemes and FFT acceleration for the evaluation of modal kernels \cite{article26,article28,article25}. 

However, many geometries of practical interest are not exactly axisymmetric. Instead, they are quasi-axisymmetric: a local cross section rotates around a curved centerline rather than a straight axis. Figure \ref{quasiaxisymmetric} illustrates such a geometry, which arises in applications such as neural spirals, coiled waveguides, and optical or acoustic fibers. Quasi-axisymmetric structures preserve periodicity in the azimuthal direction, but they lose the global rotational symmetry that makes the axisymmetric case so convenient. As a result, the azimuthal modes are no longer decoupled, the modal kernels depend on the full geometry, and traditional dimension-reduction arguments do not apply directly. This combination of retained periodicity and broken symmetry makes quasi-axisymmetric scattering  numerically more challenging compared to the axisymmetric case. Similar geometries have been considered in other physical contexts as well. In particular, for the Stokes problem on slender bodies, \cite{article5} developed a high-order Nystr\"om framework based on precomputed quadrature rules adapted to modal kernels. The work demonstrates that it is possible to preserve high-order accuracy even when exact axisymmetry is absent, provided that the discretization exploits the remaining geometric structure. Nevertheless, for acoustic scattering by quasi-axisymmetric objects, especially in the multiple scattering setting, a comparable high-order boundary integral framework has remained unavailable.

Therefore, the objective of this paper is to develop a high-order boundary integral solver for acoustic scattering by quasi-axisymmetric structures and then extend it systematically to multiple scattering. The starting point is a single-body Nystr\"om discretization that exploits azimuthal periodicity through trigonometric interpolation and modal Green's functions. Singular and nearly singular self-interactions are handled by a combination of kernel splitting, FFT-based modal evaluation, and generalized Gaussian quadrature. The multiple scattering solver is then built by coupling the single-body discretizations through inter-body boundary integral interactions, so that the difficult self-interaction blocks can be treated accurately while the smooth inter-body blocks can be assembled efficiently and reused when geometric repetition is present.

In particular, we show how the modal Green's function that is standard for axisymmetric scattering can be generalized to quasi-axisymmetric geometries without relying on full rotational symmetry. We also develop a high-order Nystr\"om discretization that remains effective for both the single scattering and multiple scattering problems. In addition, we provide a multiple scattering formulation that is built directly from the single-scatterer solver, thereby preserving high-order accuracy while exposing block structure that can be exploited computationally. We also establish a convergence analysis for the proposed method on smooth geometries. These developments are demonstrated by numerical experiments for solving scattering problems involving up to 1000 quasi-axisymmetric scatterers.


\begin{figure}[!t]
	\centering
    \vspace{-0.5cm}
	\includegraphics[width=5.5in]{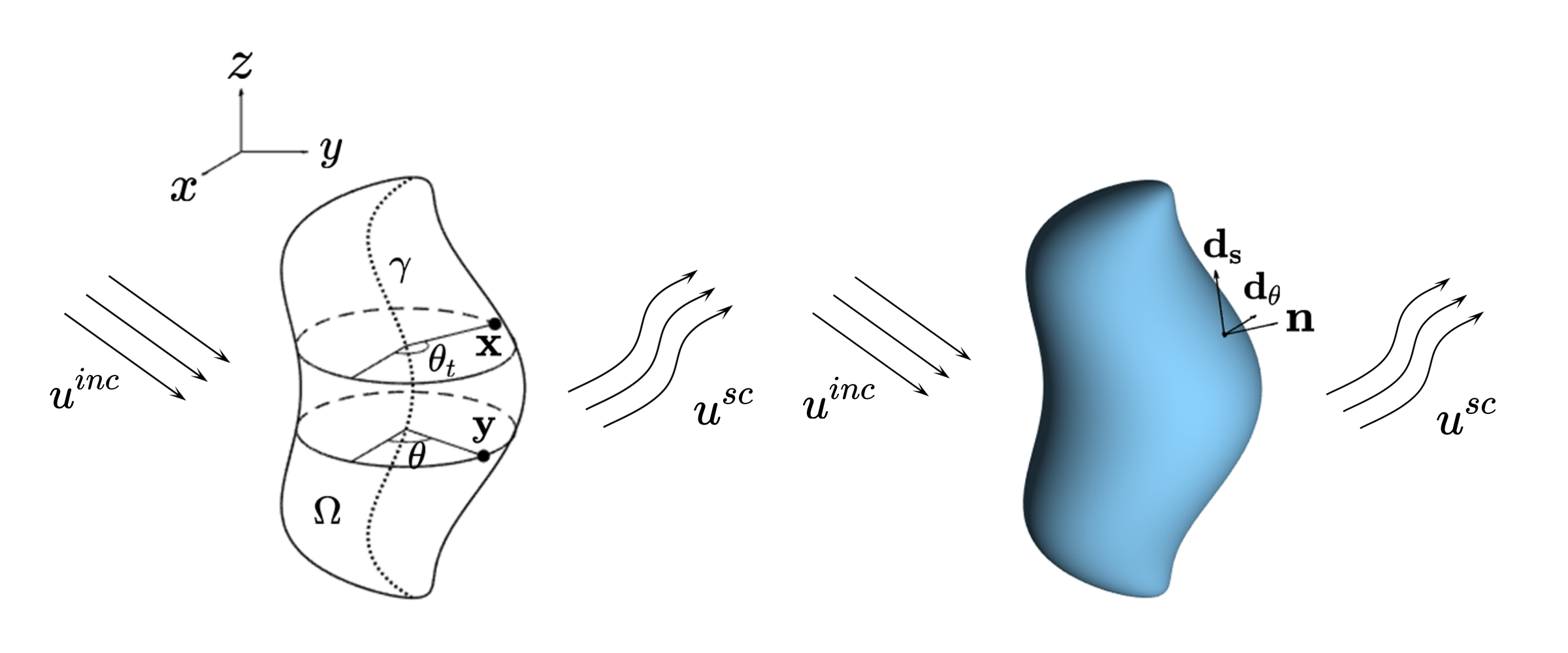}
    \vspace{-0.2cm}
\caption{Geometry of quasi-axisymmetric structures}
	\label{quasiaxisymmetric}
\end{figure}

The remainder of this paper is organized as follows. Section 2 presents the boundary integral formulation of the problem. Section 3 describes the Nystr\"om discretization and the associated numerical algorithms for quasi-axisymmetric structures. Section 4 extends the single scattering framework to multiple scattering problems. Section 5 provides a convergence analysis for smooth geometries. Section 6 presents numerical results and a discussion of multiple scattering by scatterers under random rotations and translations. Finally, Section 7 concludes the paper.

\section{Integral equation formulations}
Consider a bounded quasi-axisymmetric obstacle $\Omega\subset\mathbb{R}^3$ with a smooth boundary $\partial\Omega$ parameterized by
\begin{equation}\label{para_quasi}
    \partial\Omega := \left( r(s)\cos\theta+p(s), r(s)\sin\theta+q(s), z(s) \right)
\end{equation}
where $s \in [0,T]$, $\theta \in [0,2\pi]$. 
We refer to such a geometry as quasi-axisymmetric, as it can be generated by sweeping a circular profile of radius $r(s)$ along the generating centerline $\gamma = (p(s),q(s),z(s))$. 
While we assume a smooth boundary to simplify the discussion, our numerical approach remains applicable to non-smooth geometries. 
We focus on the exterior scattering problem under sound-soft boundary conditions, as the extension to other cases is mathematically straightforward. 
Assume that the obstacle is embedded in a homogeneous medium and is illuminated by a time-harmonic incident wave $u^{inc}$, inducing a scattering field $u^{sc}$. 
In this case, the total field $u=u^{inc}+u^{sc}$ satisfies the Helmholtz equation:
\begin{equation}\label{helmholtz_equ}
\left\{
\begin{aligned}
& \Delta u+k^2u=0 \mbox{ in } \mathbb{R}^3\backslash \overline{\Omega},\\
& u=0 \mbox{ on } \partial\Omega,
\end{aligned}
\right.
\end{equation}
where $k>0$ is the wavenumber of the background medium. 
The incident field $u^{inc}$ is a time harmonic plane wave $u^{inc}=e^{ik\mathbf{d}\cdot\mathbf{x}}$, where the vector $\mathbf{d}\in \mathbb{R}^3$ is the direction of propagation. The scattering field $u^{sc}$ satisfies the Sommerfeld radiation condition:
\begin{equation}\label{sommer_rad}
\lim_{r\to\infty}
r
\left(
\frac{\partial u^{sc}}{\partial r}
-
iku^{sc}
\right)
=0,
\end{equation}
where $r=|\mathbf{x}|$.

We first introduce two boundary integral operators. Let $n(\mathbf{x})$ be the unit exterior normal vector on $\partial\Omega$. We refer to $S$ and $D$ as the single-layer and double-layer potential operators, respectively, which are defined as
\begin{eqnarray}
    S[\mu](\mathbf{x}):&=&\int_{\partial \Omega}G^k(\mathbf{x},\mathbf{y})\mu(\mathbf{y})ds(\mathbf{y}),\quad \mathbf{x}\in \mathbb{R}^3\backslash\partial\Omega, \\
    D[\mu](\mathbf{x}):&=&\int_{\partial \Omega}\frac{\partial G^k(\mathbf{x},\mathbf{y})}{\partial n(\mathbf{y})}\mu(\mathbf{y})ds(\mathbf{y}),\quad \mathbf{x}\in \mathbb{R}^3\backslash\partial\Omega,
\end{eqnarray}
where $\mu(\mathbf{x})\in C(\partial \Omega)$ is the density function, and $G^k(\mathbf{x},\mathbf{y})$ is the fundamental solution of the three-dimensional Helmholtz equation:
\begin{equation}
\label{Green}
G^k(\mathbf{x},\mathbf{y})=\frac{e^{ik\left|\mathbf{x}-\mathbf{y} \right|}}{4\pi\left|\mathbf{x}-\mathbf{y} \right|}.
\end{equation}
Note that both $S[\mu]$ and $D[\mu]$ are solutions of the Helmholtz equation and automatically satisfy the Sommerfeld radiation condition \eqref{sommer_rad}. When $\mathbf{x} \in \mathbb{R}^3\backslash \overline{\Omega}$ approaches a point $\mathbf{x}_0$ on the boundary $\partial\Omega$, the single- and double-layer potentials satisfy the jump property \cite{article13}:
\begin{eqnarray}
    \label{jump1}    \lim\limits_{\mathbf{x}\to\mathbf{x}_0}S[\mu](\mathbf{x})&=&\int_{\partial \Omega}G^k(\mathbf{x}_0,\mathbf{y})\mu(\mathbf{y})ds(\mathbf{y})=\mathcal{S}[\mu][\mathbf{x}_0],\\
     \label{jump3}   \lim\limits_{\mathbf{x}\to\mathbf{x}_0}D[\mu](\mathbf{x})&=&\frac{1}{2}\mu(\mathbf{x}_0)+\int_{\partial \Omega}\frac{\partial G^k(\mathbf{x}_0,\mathbf{y})}{\partial n(\mathbf{y})}\mu(\mathbf{y})ds(\mathbf{y}) = \frac{1}{2}\mu(\mathbf{x}_0)+\mathcal{D}[\mu][\mathbf{x}_0],
\end{eqnarray}
where $\mathcal{S}$ and $\mathcal{D}$ are the corresponding boundary operators for the single-layer and double-layer potentials.
By potential theory, we can represent the scattered field $u^{sc}$ using either the single-layer potential $S[\mu]$ or the double-layer potential $D[\mu]$. Using the boundary condition on $\partial \Omega$, the integral formulation $u^{sc}=S[\mu]$ leads to a Fredholm integral equation of the first kind:
\begin{equation}\label{single-layer-bie}
\mathcal{S}[\mu][\mathbf{x}] = -u^{inc}(\mathbf{x}),\quad \mathbf{x}\in \partial\Omega. 
\end{equation} 
The double-layer potential formulation $u^{sc}=D[\mu]$ leads to a Fredholm integral equation of the second kind:
\begin{equation}\label{double-layer-bie}
\left(\frac{1}{2}+\mathcal{D}\right)[\mu][\mathbf{x}]=-u^{inc}(\mathbf{x}), \quad \mathbf{x}\in \partial\Omega.
\end{equation} 

In the following sections, we discuss discretizations for these two formulations on quasi-axisymmetric geometries.
\begin{remark}
    For simplicity, we assume that there is no resonance for the interior boundary value problem \eqref{helmholtz_equ}. If resonance occurs, one can switch to the combined layer potential~\cite{article41}
    \begin{eqnarray}
        u^{sc}(\mathbf{x}) = (D-i S) [\mu](\mathbf{x}).
    \end{eqnarray}
   Its discretization follows directly from the discretization of \eqref{single-layer-bie} and \eqref{double-layer-bie}.
\end{remark}  
 

\section{Discretizations of boundary integral equations}
\subsection{Single-layer formulation}
For a target point $\mathbf{x}$ and a source point $\mathbf{y}$ on $\Gamma$, let
$$
\begin{cases}
	\begin{aligned}
		\mathbf{x}&=(x_{t},y_{t},z_{t})=(r(t)\cos(\theta_{t})+p(t),r(t)\sin(\theta_{t})+q(t),z(t)),\\
		\mathbf{y}&=(x,y,z)=(r(s)\cos(\theta)+p(s),r(s)\sin(\theta)+q(s),z(s)),
	\end{aligned}
\end{cases}
$$
where $t,s\in [0,T]$ and $\theta_t,\theta\in [0,2\pi]$.
Denote $\mathbf{d}_s(\mathbf{y})$ and $\mathbf{d}_\theta(\mathbf{y})$ the tangent vectors in the polar direction $s$ and the azimuthal direction $\theta$, respectively:
\begin{eqnarray*}
	\mathbf{d}_s(\mathbf{y}) &=&        
	\left(
	r'(s)\cos(\theta) +p'(s),
	r'(s)\sin(\theta) +q'(s) ,
	z'(s) 
	\right),\\
	\mathbf{d}_\theta(\mathbf{y}) &=&     
	\left(
	-r(s)\sin(\theta) ,
	r(s)\cos(\theta) ,
	0 
	\right).
\end{eqnarray*}
Using the parameterization \eqref{para_quasi} of $\partial\Omega$ over $[0,T]\times[0,2\pi]$, the surface integral \eqref{single-layer-bie} can be written as the double integral
\begin{equation}\label{smu}
   \int_{\partial \Omega}G^k(\mathbf{x},\mathbf{y})\mu(\mathbf{y})ds(\mathbf{y})
    =\int_0^T\int_0^{2\pi}\frac{e^{ik|\mathbf{x}-\mathbf{y}|}}{4\pi|\mathbf{x}-\mathbf{y}|}\mu(\mathbf{y})|\mathrm{J}(\mathbf{y})|d\theta ds,
\end{equation}
where $
|\mathrm{J}(\mathbf{y})|
=|\mathbf{d}_s(\mathbf{y})\times \mathbf{d}_\theta (\mathbf{y})|
$ is the Jacobian factor. 

To evaluate the integral \eqref{smu} numerically, we partition $s\in[0,T]$ into $N_P$ panels, denoted by $\gamma_1,\cdots,\gamma_{N_P}$, and discretize each panel using $N_{\mathrm{order}}$ Gauss-Legendre nodes. In the azimuthal direction, we use a uniform grid of $2N$ points,
\begin{eqnarray*}
\theta_j=\frac{j\pi }{N},\quad j=0,1,\cdots,2N-1.
\end{eqnarray*}
Let $\mathbf{x}$ be a target point located in the $p$th panel, so that $\mathbf{x}$ lies in the parameter domain $\gamma_p\times[0,2\pi]$. When $\mathbf{x}$ and the source point $\mathbf{y}$ in \eqref{smu} are well separated, we apply the trapezoidal rule in the azimuthal direction and composite Gauss-Legendre quadrature in the polar direction. Specifically,
\begin{eqnarray}
S[\mu](\mathbf{x})&\approx&\sum\limits_{i=1,i\neq p-1,p,p+1}^{N_P}\sum\limits_{l=1}^{N_{order}}\sum\limits_{j=0}^{2N-1}\frac{e^{ik|\mathbf{x}-\mathbf{y}_{ijl}|}}{4\pi|\mathbf{x}-\mathbf{y}_{ijl}|}\mu(\mathbf{y}_{ijl})|\mathrm{J}(\mathbf{y}_{ijl})|\frac{\pi}{N}\omega_{il}\nonumber\\
\label{sx}
&&+\int_{\gamma_{p-1}\cup \gamma_{p}\cup\gamma_{p+1}}\int_0^{2\pi}\frac{e^{ik|\mathbf{x}-\mathbf{y}|}}{4\pi|\mathbf{x}-\mathbf{y}|}\mu(\mathbf{y})|\mathrm{J}(\mathbf{y})|d\theta ds,
\end{eqnarray}
where $\omega_{il}$ is the rescaled Gaussian weight corresponding to the $l$th Gauss-Legendre node in the $i$th panel. When $\mathbf{y}$ lies on the panels $\gamma_{p-1}$, $\gamma_p$, or $\gamma_{p+1}$, the second integral on the right-hand side of \eqref{sx} becomes singular or nearly singular. Applying the same quadrature rule as in the well-separated case would therefore introduce substantial error, so these three panels must be treated separately.


We first consider the integral over the panel $\gamma_p$, and the treatment of the two adjacent panels is analogous. By interpolation on Gauss-Legendre nodes in $s$ and equally spaced points in $\theta$, we obtain the approximation
\begin{eqnarray}
\mu(\mathbf{y})|\mathrm{J}(\mathbf{y})|\approx\sum\limits_{l=1}^{N_{order}}\sum\limits_{j=0}^{2N-1}|\mathrm{J}(\mathbf{y}_{pjl})|\mu(\mathbf{y}_{pjl})P_l(s)L_j(\theta), 
\end{eqnarray}
where $P_l(s)$ is the $l$th scaled Lagrange basis polynomial associated with the Gauss-Legendre nodes on $\gamma_p$, and $L_j(\theta)$ is the $j$th trigonometric Lagrange basis function,
\begin{eqnarray*}
L_j(\theta)=\frac{1}{2N}\left(1+2\sum\limits_{m=1}^{N-1}\cos m(\theta-\theta_j)+\cos N(\theta-\theta_j)\right),\quad \theta\in[0,2\pi],\ j=0,1,\cdots,2N-1,
\end{eqnarray*}
or, equivalently \cite{article13},
\begin{eqnarray*}
L_j(\theta)=\begin{cases}\frac{1}{2N}\sin N(\theta-\theta_j)\cot \frac{\theta-\theta_j}{2},\quad \theta\ne \theta_j,\\
1,\quad \theta= \theta_j. \end{cases}
\end{eqnarray*}
Accordingly, the integral over $\gamma_p$ can be approximated by
\begin{eqnarray}
&&\int_{\gamma_{p}}\int_0^{2\pi}\frac{e^{k|\mathbf{x}-\mathbf{y}|}}{4\pi|\mathbf{x}-\mathbf{y}|}\mu(\mathbf{y})|\mathrm{J}(\mathbf{y})|d\theta ds\nonumber\\
\label{muJ}
&\approx&\sum\limits_{l=1}^{N_{order}}\sum\limits_{j=0}^{2N-1}\mu(\mathbf{y}_{pjl})|\mathrm{J}(\mathbf{y}_{pjl})|\int_{\gamma_p}P_l(s)\int_0^{2\pi}\frac{e^{ik|\mathbf{x}-\mathbf{y}|}}{4\pi|\mathbf{x}-\mathbf{y}|}L_j(\theta)d\theta ds.
\end{eqnarray}
The problem is thus reduced to the efficient and accurate evaluation of
\begin{eqnarray}
\label{coef_ele}
\int_{\gamma_p}P_l(s)\int_0^{2\pi}\frac{e^{ik|\mathbf{x}-\mathbf{y}|}}{4\pi|\mathbf{x}-\mathbf{y}|}L_j(\theta)d\theta ds.
\end{eqnarray}
It turns out the inner integral can be evaluated in essentially the same manner as in axisymmetric scattering \cite{article26, article25}, namely through the computation of modal Green's functions, defined by
\begin{eqnarray}\label{modalgreen}
G_m^{(k)}=\int_0^{2\pi}\frac{e^{ik|\mathbf{x}-\mathbf{y}|}}{4\pi|\mathbf{x}-\mathbf{y}|}e^{-im\theta}d\theta,\quad m=-N,\cdots,0,1,\cdots,N-1.
\end{eqnarray}
A detailed discussion on the evaluation of equation \eqref{modalgreen} is given in subsection \ref{subsec:modalgreen}.  If these modal Green's functions can be evaluated efficiently, then the inner integral in \eqref{coef_ele} is simply
\begin{eqnarray*}
\int_0^{2\pi}\frac{e^{ik|\mathbf{x}-\mathbf{y}|}}{4\pi|\mathbf{x}-\mathbf{y}|}L_j(\theta)d\theta
=\left\{\mathcal{D}^{-1}[G_m^{(k)}]\right\}_{j},\quad j=0,\cdots,2N-1,
\end{eqnarray*}
where $\mathcal{D}^{-1}$ denotes the inverse $\mathrm{DFT}$.

\subsection{Double-layer formulation}

Using the jump relation \eqref{jump3}, the integral equation associated with the double-layer potential can be written as
\begin{equation}
\label{double}
\frac{1}{2}\mu(\mathbf{x})
+
\int_{\partial \Omega}
\frac{\partial G^k(\mathbf{x},\mathbf{y})}
{\partial n(\mathbf{y})}
\mu(\mathbf{y})ds(\mathbf{y})
=
-u^{inc}(\mathbf{x}).
\end{equation}
Denote by $\{e_r,e_\theta,e_z\}$ the cylindrical coordinate basis. Since the unit normal vector is orthogonal to all tangent vectors on the surface, we have
$n\cdot e_\theta=0$. Hence it can be written as $n=n_r e_r+n_z e_z$. 
The gradient operator admits the following representation in cylindrical coordinates
\begin{equation}
\nabla_y = e_r\frac{\partial}{\partial r} + e_\theta\frac1r\frac{\partial}{\partial\theta} + e_z\frac{\partial}{\partial z}.
\end{equation}
Therefore, by the definition of the normal derivative and using the orthogonality relations of the cylindrical basis vectors, we obtain
\begin{equation} \label{normal-derivative}
\frac{\partial G^k(\mathbf x,\mathbf y)}
{\partial n(\mathbf y)}
=
\nabla_y G^k(\mathbf x,\mathbf y)
\cdot n(\mathbf y)
=
n_r\frac{\partial G^k}{\partial r}+n_z\frac{\partial G^k}{\partial z}.
\end{equation}
Let $\rho=|\mathbf{x}-\mathbf{y}|$ be the distance between the target and source points, and differentiate $G^k$ with respect to $\rho$ yields
\begin{equation}
\label{dgdrho}
\frac{dG^k}{d\rho}
= \frac{e^{ik\rho}(ik\rho-1)}{4\pi\rho^2}.
\end{equation}
For convenience, denote $r=r(s)$ and $z=z(s)$. It holds
\begin{equation}
\label{rho2}
\rho^2
=|\mathbf{x}-\mathbf{y}|
=a^2+b^2+r^2+c^2-2r(a\cos\theta+b\sin\theta),
\end{equation}
where
\[
a=x_t-p(s),\quad
b=y_t-q(s),\quad
c=z_t-z(s),
\]
Let
$\tilde r_t=\sqrt{a^2+b^2},
\quad
\alpha=\arg(a+ib)$.
Then, we have
\[
a=\tilde r_t\cos\alpha,
\quad
b=\tilde r_t\sin\alpha,
\]
which implies
\[
a\cos\theta+b\sin\theta
=
\tilde r_t\cos(\theta-\alpha).
\]
Introducing the shifted angular variable $\phi=\theta-\alpha$, equation \eqref{rho2} can be rewritten as
\begin{equation}
\label{rho-standard}
\rho^2=\tilde r_t^2+r^2+c^2-2r\tilde r_t\cos\phi.
\end{equation}
Differentiating equation \eqref{rho-standard} with respect to the source variables gives
\[
\partial_r\rho =\frac{r-\tilde r_t\cos\phi}{\rho},
\quad
\partial_z\rho=-\frac{z_t-z}{\rho}.
\]
Using the chain rule together with equation \eqref{dgdrho}, we obtain
\begin{equation}
\partial_r G^k
=
\frac{e^{ik\rho}(ik\rho-1)}{4\pi\rho^3}(r-\tilde r_t\cos\phi),\quad
\partial_z G^k=-\frac{e^{ik\rho}(ik\rho-1)}{4\pi\rho^3}(z_t-z).
\end{equation}
The corresponding Fourier modal Green's functions are defined by
\[
\partial_r \hat{G}_m^{(k)}
=
\int_0^{2\pi}
\frac{\partial G^k(\mathbf{x},\mathbf{y})}
{\partial r}
e^{-im\theta}\,d\theta,
\quad
\partial_z \hat{G}_m^{(k)}
=
\int_0^{2\pi}
\frac{\partial G^k(\mathbf{x},\mathbf{y})}
{\partial z}
e^{-im\theta}\,d\theta.
\]
By equation \eqref{normal-derivative}, it holds
\begin{equation}
\label{modal-normal}
\partial_n \hat{G}_m^{(k)}
=
n_r\partial_r \hat{G}_m^{(k)}
+
n_z\partial_z \hat{G}_m^{(k)}.
\end{equation}
Consequently, the double-layer potential can also be represented through the modal Green's functions and their source-point derivatives. They are discretized using the periodic trapezoidal rule on equispaced nodes and evaluated efficiently via the Fast Fourier Transform (FFT), as in the single-layer formulation case.

 In the next subsection, we discuss how to evaluate the modal Green's functions together with their derivative modal quantities efficiently using recurrence formulas and FFT-based convolution techniques. 

\subsection{Evaluation of the modal Green's functions}\label{subsec:modalgreen}
Modal Green's functions have a long history in boundary integral methods for bodies of revolution \cite{article26,article59,article25}. In the present quasi-axisymmetric setting, global rotational symmetry is lost, but these modal Green's functions remain useful locally because the surface retains an azimuthal periodic parameter.  Many approaches have been proposed for the efficient evaluation of modal Green's functions, and we refer readers to \cite{article58} for a comprehensive discussion.  Here we adopt the technique based on recursion formulas and convolution \cite{article25} to evaluate \eqref{modalgreen} when $\rho=|\mathbf{x}-\mathbf{y}|$ is small. 

We first consider the case $k=0$. 
Recall from equation \eqref{rho-standard} that the distance function depends on the angular variable through the term \(2r\tilde r_t\cos\phi\). Factoring out \(2r\tilde r_t\), we rewrite
\[
\rho^2
=
2r\tilde r_t
\left(
\frac{
\tilde r_t^2+r^2+c^2
}{
2r\tilde r_t
}
-
\cos\phi
\right).
\]
Defining
$\chi
=
\frac{
\tilde r_t^2+r^2+c^2
}{
2r\tilde r_t
}$, the distance function admits the representation
\[
\rho
=
\sqrt{2r\tilde r_t}
\sqrt{\chi-\cos\phi}.
\]
This representation extracts the angular singularity of the kernel into the canonical factor $(\chi-\cos\phi)^{-1/2}$, which admits an explicit representation in terms of Legendre functions.

In other words, it holds
\begin{eqnarray}\label{modal0}
\int_0^{2\pi}\frac{1}{4\pi\rho}\cos m(\theta-\theta_j)d\theta 
&=&\frac{1}{2\pi\sqrt{\tilde{r_t}r}}\int_0^{2\pi} \frac{\cos m(\theta-\theta_j)}{\sqrt{8(\chi-\cos(\theta-\alpha))}}d\theta\notag\\
&=&\frac{\cos(m(\alpha-\theta_j))}{2\pi\sqrt{\tilde{r_t}r}}\int_0^{2\pi} \frac{\cos m\theta}{\sqrt{8(\chi-\cos\theta)}}d\theta\notag\\
&=&\frac{\cos(m(\alpha-\theta_j))}{2\pi\sqrt{\tilde{r_t}r}}\mathcal{Q}_{m-\frac{1}{2}}(\chi),
\end{eqnarray}
where ${Q}_{m-\frac{1}{2}}$ is the Legendre function of the second kind of half-degree \cite{article37}. It is well known that ${Q}_{m-\frac{1}{2}}$ and its derivative can be evaluated through the recursion formulas:
\begin{eqnarray} 
\label{leg_recur}
\mathcal{Q}_{m-1/2}(\chi)&=&4\frac{m-1}{2m-1}\chi\mathcal{Q}_{m-3/2}(\chi)-\frac{2m-3}{2m-1}\mathcal{Q}_{m-5/2}(\chi),\\
\label{dleg_recur}
\frac{\partial \mathcal{Q}_{m-1/2}(\chi)}{\partial \chi}&=&\frac{2m-1}{2(\chi^2-1)}(\chi\mathcal{Q}_{m-1/2}(\chi)-\mathcal{Q}_{m-3/2}(\chi)),
\end{eqnarray}
with 
\begin{eqnarray*}
    \mathcal{Q}_{-1/2}(\chi)&=&\mu K(\mu),\\
\mathcal{Q}_{1/2}(\chi)&=&\chi\mu K(\mu)-\sqrt{2(\chi+1)}E(\mu),
\end{eqnarray*}
where $\mu=\sqrt{2/(\chi+1)}$, and $K$ and $E$ are the complete elliptic integrals of the first and second kinds, respectively \cite{article37}. 

To evaluate $G_m^{(k)}$ for $k\neq0$, we define
\begin{eqnarray*}
f_1(\theta)=\cos(k\rho (\theta)),\quad f_2(\theta)=\frac{\sin(k\rho (\theta))}{4\pi\rho (\theta)},\quad
g(\theta)=\frac{1}{4\pi\rho (\theta)}.
\end{eqnarray*}
Then $G_m^{(k)}$ can be decomposed into two parts:
\begin{eqnarray}
\label{conv}
\int_0^{2\pi}\frac{e^{ik\rho}}{4\pi\rho}e^{-im\theta}d\theta
&=&\int_0^{2\pi}f_1(\theta)g(\theta)e^{-im\theta}d\theta+i\int_0^{2\pi}f_2(\theta)e^{-im\theta}d\theta\notag \\&=&I_1+iI_2.
\end{eqnarray}
After removing the removable singularity at $\rho=0$ and using a Taylor expansion, we see that $f_2(\theta)$ is analytic in $\theta$. Therefore, $I_2$ can be evaluated directly on the equally spaced grid $\{\theta_j\}$ and accelerated efficiently by the FFT.

To evaluate $I_1$, we apply a convolution formula together with the modal Green's function \eqref{modal0}. Specifically,
\begin{eqnarray}\label{I1}
I_1
&=&\frac{1}{2\pi}\sum\limits_{n=-\infty}^{\infty}\hat{f}_{1,n}\hat{g}_{m-n},
\end{eqnarray}
where $\hat{f}_{1,n}$ and $\hat{g}_n$ are the Fourier coefficients of $f_1$ and $g$, respectively:
\begin{eqnarray}
\hat{f}_{1,n}=\int_0^{2\pi}f_1(\theta)e^{-in\theta}d\theta,\quad
\hat{g}_n=\int_0^{2\pi}g(\theta)e^{-in\theta}d\theta.
\end{eqnarray} 
The summation in \eqref{I1} converges rapidly, as $\hat{f}_{1,n}$ decays very fast. This completes the evaluation of the modal Green's function. 

For the derivatives of the modal Green's functions, we again consider the static case $k=0$. Recall that
\[
\hat{G}_m^{(0)}
=
\frac{
Q_{m-\frac12}(\chi)
}{
2\pi\sqrt{\tilde r_t r}
}.
\]
Here $\tilde r_t$ is fixed in the local modal evaluation, while $r$ and $z$ are the source variables. Differentiating with respect to the source variables gives
\[
\partial_r\hat{G}_m^{(0)}
=
\frac{
Q'_{m-\frac12}(\chi)
}{
2\pi\sqrt{\tilde r_t r}
}
\partial_r\chi
-
\frac1{2r}
\hat{G}_m^{(0)},
\quad
\partial_z\hat{G}_m^{(0)}
=
\frac{
Q'_{m-\frac12}(\chi)
}{
2\pi\sqrt{\tilde r_t r}
}
\partial_z\chi.
\]
where
\[
\partial_r\chi
=
\frac{
r^2-\tilde r_t^2-c^2
}{
2\tilde r_t r^2
},
\quad
\partial_z\chi
=
-\frac{
z_t-z
}{
r\tilde r_t
}.
\]
The derivative of the Legendre function can be evaluated by \eqref{dleg_recur}. For $k\neq0$, the derivative kernels are again decomposed into convolutions between the static modal kernels and smooth oscillatory correction factors. After removing the removable singularities by Taylor expansion, the remaining smooth factors are evaluated on an equispaced grid and combined by FFT-based convolutions. Consequently, the derivative modal quantities
$\partial_r\hat{G}_m^{(k)}$, $\partial_z\hat{G}_m^{(k)}$ and $\partial_n\hat{G}_m^{(k)}$ can be evaluated efficiently in the same modal framework. Further discussion on this approach, including how to overcome the stability issue, can be found in \cite{article26}.

\subsection{Evaluation of the outer integral}
After the inner integral in \eqref{coef_ele} has been evaluated, the remaining task is to compute the outer integral
\begin{eqnarray*}
\tilde{I}_{lj}(\mathbf{x}) := \int_{\gamma_p} P_l(s) G_j(t,s)\, ds,
\mbox{ with }
G_j(t,s) = \int_0^{2\pi} \frac{e^{ik|\mathbf{x}-\mathbf{y}|}}{4\pi|\mathbf{x}-\mathbf{y}|} L_j(\theta)\, d\theta.
\end{eqnarray*}
The function $G_j(t,s)$ is weakly singular when the target point lies on or near the source panel, so this integral is evaluated using precomputed high-accuracy generalized Gaussian quadrature \cite{article48}. The number of quadrature nodes depends on the target location. The construction of such quadrature nodes was obtained through nonlinear optimization  with code available at \url{github.com/JamesCBremerJr/GGQ}. A detailed scheme for applying these generalized quadratures to the single-layer boundary operator can be found in \cite{article59}. The extension to the double-layer case is analogous. In the numerical examples below, we use a 16th-order generalized Gaussian rule, with $16$ target-dependent nodes and weights for the self panel and $48$ target-independent nodes and weights for adjacent panels.
This procedure yields the discretized matrix for the integral equation, reducing the scattering problem to a linear system of size $N_{\mathrm{tot}}=N_p\times N_{\mathrm{order}}\times (2N)$.

\section{Extension to multiple scattering}
\label{multiple_objects}
The algorithms described in the previous section solve boundary integral equations posed on a single quasi-axisymmetric surface in $\mathbb{R}^3$. In this section, we extend the framework to multiple scatterers. In contrast to the single body case, the formulation must now account for both self-interaction and inter-body scattering.

We consider multiple objects that are identical in shape and size. Extension to different shapes or sizes will be discussed later. Without loss of generality, suppose that there are two geometries, $\Omega_1$ and $\Omega_2$, which are congruent under rigid motions (rotations and translations). We further assume that the objects are well separated, so that near-field interactions do not occur:
\begin{eqnarray*}
\min\limits_{\mathbf{x}\in\Omega_1,\mathbf{y}\in\Omega_2}\{|\mathbf{x}-\mathbf{y}|\}>=\alpha,
\end{eqnarray*}
where $\alpha>0$ is a constant greater than half the wavelength.
For the exterior Dirichlet problem, we represent the scattered field by the double-layer potential
\begin{eqnarray}
	u(\mathbf{x})=\int_{\partial \Omega_1}\frac{\partial G^{k}(\mathbf{x},\mathbf{y})}{\partial n(\mathbf{y})}\mu (\mathbf{y})ds(\mathbf{y})+\int_{\partial \Omega_2}\frac{\partial G^{k}(\mathbf{x},\mathbf{y})}{\partial n(\mathbf{y})}\mu (\mathbf{y})ds(\mathbf{y}), 
\end{eqnarray}
and obtain the following boundary integral equation:
\begin{eqnarray}
	\label{multi-double}
	\frac{1}{2}\mu(\mathbf{x})+\int_{\partial \Omega_1\cup\partial\Omega_2}\frac{\partial G^{k}(\mathbf{x},\mathbf{y})}{\partial n(\mathbf{y})}\mu (\mathbf{y})ds(\mathbf{y})=-u^{inc}(\mathbf{x}),\quad \mathbf{x}\in \partial \Omega_1\cup\partial\Omega_2.
\end{eqnarray}
When the target point $\mathbf{x}$ lies on $\partial\Omega_1$, the first contribution is
\begin{eqnarray*}
	\int_{\partial \Omega_1}\frac{\partial G^k(\mathbf{x},\mathbf{y})}{\partial n(\mathbf{y})}\mu(\mathbf{y})ds(\mathbf{y})=
	\int_0^T\int_0^{2\pi}\frac{\partial G^k(\mathbf{x},\mathbf{y})}{\partial n(\mathbf{y})}\mu(\mathbf{y})|\mathrm{J}(\mathbf{y})|d\theta ds.
\end{eqnarray*}
Its discretization is exactly the same as in the single body case. 
For the second part,
\begin{eqnarray}
\label{multi2part}
\int_{\partial \Omega_2}\frac{\partial G^{k}(\mathbf{x},\mathbf{y})}{\partial n(\mathbf{y})}\mu (\mathbf{y})ds(\mathbf{y}),
\end{eqnarray}
since $\mathbf{x}$ and $\mathbf{y}$ belong to different boundaries, the integrand in \eqref{multi2part} is smooth. We may therefore apply the trapezoidal rule in the azimuthal direction and Gauss-Legendre quadrature in the polar direction:
\begin{align*}
    &\int_{\partial \Omega_2}\frac{\partial G^{k}(\mathbf{x},\mathbf{y})}{\partial n(\mathbf{y})}\mu (\mathbf{y})ds(\mathbf{y})\\\approx&\sum\limits_{i=1}^{N_P}\sum\limits_{l=1}^{N_{order}}\sum\limits_{j=0}^{2N-1}\frac{\partial G^{k}(\mathbf{x},\mathbf{y}_{ilj})}{\partial n(\mathbf{y}_{ilj})}\mu (\mathbf{y}_{ilj})|\mathrm{J}(\mathbf{y}_{ilj})|\frac{\pi}{N}w_{il}\\
    =&\sum\limits_{i=1}^{N_P}\sum\limits_{l=1}^{N_{order}}\sum\limits_{j=0}^{2N-1}\left(\frac{i\kappa}{\left| \mathbf{x}-\mathbf{y}_{ilj} \right|}-\frac{1}{\left| \mathbf{x}-\mathbf{y}_{ilj} \right|^2}\right)\frac{e^{i\kappa\left| \mathbf{x}-\mathbf{y}_{ilj} \right|}}{4\pi\left| \mathbf{x}-\mathbf{y}_{ilj} \right|}{n}(\mathbf{y}_{ilj})\cdot(\mathbf{y}_{ilj}-\mathbf{x})\mu (\mathbf{y}_{ilj})|\mathrm{J}(\mathbf{y}_{ilj})|\frac{\pi}{N}w_{il}.
\end{align*}


Consequently, the multiple scattering problem can be written as the linear system
\begin{eqnarray}
\left(\frac12 \mathbf I+\mathbf D\right)\boldsymbol{\mu}
=
-\mathbf u^{inc},
\end{eqnarray}
where $\mathbf I$ is the $N_{\mathrm{tot}}\times N_{\mathrm{tot}}$ identity matrix. Here $N_{\mathrm{tot}}$ depends on both the discretization level of each geometry and the total number of scatterers. Let $N_{\mathrm{geo}}$ denote the number of geometries, and let $n_1,n_2,\cdots,n_{N_{\mathrm{geo}}}$ be the numbers of discrete points on their surfaces. Moreover, $\mathbf D$ is an $N_{\mathrm{tot}}\times N_{\mathrm{tot}}$ coefficient matrix composed of $N_{\mathrm{geo}}\times N_{\mathrm{geo}}$ blocks. The diagonal blocks represent self-interactions of individual objects, whereas the off-diagonal blocks represent interactions between distinct scatterers. The linear system can be written in block form as
\begin{eqnarray}\label{linearsys}
\frac12
\begin{bmatrix}
\boldsymbol{\mu}_1 \\
\boldsymbol{\mu}_2 \\
\vdots \\
\boldsymbol{\mu}_{N_{\mathrm{geo}}}
\end{bmatrix}
+
\begin{bmatrix}
\mathbf D_{1,1} & \mathbf D_{1,2} & \cdots & \mathbf D_{1,N_{\mathrm{geo}}} \\
\mathbf D_{2,1} & \mathbf D_{2,2} & \cdots & \mathbf D_{2,N_{\mathrm{geo}}} \\
\vdots & \vdots & \ddots & \vdots \\
\mathbf D_{N_{\mathrm{geo}},1} & \mathbf D_{N_{\mathrm{geo}},2} & \cdots & \mathbf D_{N_{\mathrm{geo}},N_{\mathrm{geo}}}
\end{bmatrix}
\begin{bmatrix}
\boldsymbol{\mu}_1 \\
\boldsymbol{\mu}_2 \\
\vdots \\
\boldsymbol{\mu}_{N_{\mathrm{geo}}}
\end{bmatrix}
=
-
\begin{bmatrix}
\mathbf u^{inc}_1 \\
\mathbf u^{inc}_2 \\
\vdots \\
\mathbf u^{inc}_{N_{\mathrm{geo}}}
\end{bmatrix}.
\end{eqnarray}
Here $\mathbf D_{t',t}$ denotes the matrix obtained from the discretization of the interaction operator
\[
D_{t',t}[\mu]
=
\int_{\partial \Omega_t}
\frac{\partial G^{k}(\mathbf{x},\mathbf{y})}{\partial n(\mathbf{y})}
\mu(\mathbf{y})\,ds(\mathbf{y}),
\quad
\mathbf{x}\in\partial\Omega_{t'}.
\]
A direct computational benefit of assuming identical geometries is that all self-interaction matrices $\mathbf D_{t,t}$ are identical. Hence, it suffices to compute $\mathbf D_{1,1}$ once and reuse it for every diagonal block, leading to a substantial reduction in assembly cost.

In particular, when the scatterers are arranged in a regular $P\times Q\times R$ array with uniform spacing $\delta$, the interaction matrices inherit a multilevel block Toeplitz structure. Because the scatterers are identical and the Green's function depends only on relative position, the interaction block between scatterers $t'$ and $t$ satisfies
\begin{equation}
\mathbf D_{t',t}
=
\mathbf D_{\mathbf d},
\quad
\mathbf d=\mathbf c_{t'}-\mathbf c_t,
\end{equation}
where $\mathbf d=(d_p,d_q,d_r)^T$ is the displacement vector between the two centers. Consequently, the number of distinct off-diagonal blocks decreases from $N_{\mathrm{geo}}(N_{\mathrm{geo}}-1)$ to
\[
(2P-1)(2Q-1)(2R-1)-1.
\]
For each displacement $\mathbf d$, one representative block is computed by quadrature, and every remaining block with the same displacement is obtained by direct memory copy. The resulting block reuse ratio is
\begin{equation}
1- \frac{(2P-1)(2Q-1)(2R-1)-1}{N_{\mathrm{geo}}(N_{\mathrm{geo}}-1)}
\sim
1-\frac{2^3}{N_{\mathrm{geo}}},
\quad
P,Q,R\to\infty.
\end{equation}
which approaches $1$ as the array grows. In the large-array limit, the assembly cost for distinct interaction blocks is therefore reduced from $\mathcal O(N_{\mathrm{geo}}^2)$ to $\mathcal O(N_{\mathrm{geo}})$.

The block reuse strategy described above relies on the assumption that all scatterers are identical and arranged on a regular array with uniform spacing, so that the interaction block $\mathbf D_{t',t}$ depends only on the displacement
\[
\mathbf d=\mathbf c_{t'}-\mathbf c_t.
\]
In more general configurations, for example when the scatterers have different shapes, arbitrary orientations, or nonuniform positions, this translational invariance no longer holds, and each interaction block $\mathbf D_{t',t}$ must be computed separately. In that case, the algorithm requires the evaluation of all $N_{\mathrm{geo}}(N_{\mathrm{geo}}-1)$ off-diagonal blocks, and the multilevel block Toeplitz structure is lost. Nevertheless, as long as the scatterer geometry remains the same, the self-interaction matrices $\mathbf D_{t,t}$ are still identical for all scatterers, regardless of their orientations or positions. In our numerical experiments, we examine such general configurations by applying random rotations and random position perturbations to each scatterer (Section~\ref{sec:random}). The off-diagonal blocks are then computed individually and in parallel to offset the loss of block reuse. Furthermore, the discussion of acceleration techniques for solving the linear system \eqref{linearsys} is also given in Section 6.


\section{Error analysis}\label{sec:error}

In this section, we analyze the convergence of the proposed Nystr\"om discretization for smooth quasi-axisymmetric surfaces. The main ingredients are the regularity of the boundary integral densities, the spectral accuracy of the trigonometric interpolation in the azimuthal direction, and the high-order accuracy of the generalized Gaussian quadrature along the generating curve. Together, these estimates show that the discretized boundary operators converge rapidly under mesh refinement when the geometry and boundary data are sufficiently smooth. Here we mainly focus on the convergence result for the double-layer boundary operator. The corresponding result for the single-layer potential can also be established by using the regularization theory \cite{article13}. We omit it for simplicity.

The following lemma gives a regularity statement for the double-layer boundary integral formulation. The solvability of Helmholtz boundary integral equations may be found in \cite{article41,article50}.
\begin{lemma}
    Assume $\partial\Omega$ is smooth, the boundary data $f$ is analytic on $\partial\Omega$, and $k$ is not an eigenvalue of the interior Helmholtz  Dirichlet or Neumann problem. Then the  solution density $\mu$ is analytic on $\partial\Omega$ for the double-layer potential boundary integral equation
\begin{equation*}
\frac{1}{2}\mu(\mathbf{x})+\int_{\partial \Omega}\frac{\partial G^k(\mathbf{x},\mathbf{y})}{\partial n(\mathbf{y})}\mu(\mathbf{y})ds(\mathbf{y})=f(\mathbf{x}), \quad \mathbf{x}\in \partial\Omega.
\end{equation*}
\end{lemma}


We next recall two standard exponential convergence estimates for the interpolation and quadrature rules used in the discretization.
\begin{lemma}
\label{interp_tri}
    Let $f$ be a periodic analytic function, and let $P_n[f]$ denote its trigonometric interpolant at $n$ equally spaced nodes. Then
    \begin{eqnarray*}
        \Vert P_n[f]-f\Vert_{\infty}\leq Ce^{-n\sigma},
    \end{eqnarray*}
    where the constants $C>0$ and $\sigma>0$ depend on $f$. Thus, trigonometric interpolation of periodic analytic functions converges exponentially.
\end{lemma}

\begin{lemma}
\label{interp_leg}
    Suppose that $f$ is analytic on $[a,b]$. Let $I[f]$ denote the exact integral of $f$ over $[a,b]$, and let $Q_n[f]$ denote the corresponding $n$-point Gauss-Legendre quadrature approximation. Then
     \begin{eqnarray*}
        |Q_n[f]-I[f]|\leq Ce^{-n\sigma},
    \end{eqnarray*}
    where the constants $C>0$ and $\sigma>0$ depend on $f$.
\end{lemma}
Details on trigonometric interpolation can be found in \cite{article13}, while the error analysis for Gauss--Legendre quadrature is given in \cite{article42,article43}. Together, these two lemmas establish exponential convergence for the trigonometric interpolation in the azimuthal variable and the Gauss-Legendre quadrature along the generating curve. 
\begin{theorem}
\label{thm:quad}
    Let $f=f(s,\theta)$ be analytic in $s$ and $\theta$ and periodic in $\theta$. Define the product quadrature error by
    \begin{eqnarray*}
E_{n_s,n_{\theta}}^2[f]=\int_0^T\int_{0}^{2\pi}f(s,\theta)\,d\theta\,ds-
\sum\limits_{l=1}^{n_s}\sum\limits_{j=0}^{n_{\theta}}f(s_l,\theta_j)\omega_l^s\omega_{j}^{\theta},
    \end{eqnarray*}
    where $\omega_l^s$ are the Gauss-Legendre weights in $s$, and $\omega_j^{\theta}$ are the integrals of the trigonometric Lagrange basis functions on $[0,2\pi)$. Then
    \begin{eqnarray*}
        |E_{n_s,n_{\theta}}^2[f]|\leq Ce^{-n\sigma},
    \end{eqnarray*}
    where $n=\max(n_s,n_{\theta})$, and the constants $C>0$ and $\sigma>0$ depend on $f$.
\end{theorem}

The following theorem establishes convergence for a single scatterer under the assumption that the generalized Gaussian quadrature used for singular and nearly singular panel interactions is exponentially accurate for analytic densities. This is a reasonable assumption based on the design of generalized Gaussian quadrature \cite{article48}. In particular, after the local singular part of the Helmholtz kernel is separated, the remaining panel dependent factors are analytic for smooth quasi-axisymmetric geometries, while the singular factor belongs to the finite collection of kernel functions for which the generalized Gaussian rules are constructed. Hence the corrected panel quadrature is expected to inherit exponential accuracy for analytic densities, consistent with the approximation properties of generalized Gaussian quadratures in related Nystr\"om discretizations \cite{article49}. 

\begin{theorem}[Single-body]
\label{thm:single_conv}
 Let $\mathcal{A}$ be the continuous double-layer boundary integral operator, and let $\mathcal{A}_h$ be its Nystr\"om discretization constructed from the quadrature scheme in Section~3. Here $h=(n_s,n_\theta)$, where $n_s$ is the number of quadrature points along the generating curve and $n_\theta$ is the number of azimuthal discretization points.  Let $\mu$ and $\mu_h$ solve
\begin{equation}
    \mathcal{A}\mu = f, \quad \mathcal{A}_h\mu_h = f_h,
\end{equation}
respectively. Then, for sufficiently large $n_s$ and $n_\theta$, it holds
\begin{equation}
    \|\mu - \mu_h\|_\infty \leq C e^{-n\sigma},
\end{equation}
where $n=\min(n_s,n_\theta)$, and the constants $C>0$ and $\sigma>0$ are independent of $n$.
\end{theorem}
\begin{proof}
We regard $\mathcal{A}_h$ as an operator on the analytic function space as $\mathcal{A}$ after interpolation from the discrete grid. Since the continuous boundary integral equation is uniquely solvable, $\mathcal{A}^{-1}$ is bounded. Moreover, the Nystr\"om operators $\mathcal{A}_h$ converge to $\mathcal{A}$ in the collectively compact sense \cite{article13}. Hence, for all sufficiently large $n_s$ and $n_\theta$, the operators $\mathcal{A}_h$ are invertible and their inverses are uniformly bounded, i.e.,
\begin{equation*}
    \|\mathcal{A}_h^{-1}\|_\infty \leq C_0,
\end{equation*}
with $C_0$ independent of $h$. 

Subtracting the continuous and discrete equations gives
\begin{equation*}
\mathcal{A}_h(\mu_h-\mu)=(f_h-f)+(\mathcal{A}-\mathcal{A}_h)\mu .
\end{equation*}
Therefore,
\begin{equation*}
\|\mu_h-\mu\|_\infty
\leq C_0\left(\|f_h-f\|_\infty+\|(\mathcal{A}-\mathcal{A}_h)\mu\|_\infty\right).
\end{equation*}
By Lemma~\ref{interp_tri} and Theorem~\ref{thm:quad}, the interpolation and smooth quadrature errors are exponentially small for analytic data. For the diagonal and near-diagonal panel interactions, the assumed exponential accuracy of the generalized Gaussian quadrature gives the same bound for the weakly singular contributions. Consequently,
\begin{equation*}
\|f_h-f\|_\infty+\|(\mathcal{A}-\mathcal{A}_h)\mu\|_\infty \leq C_1 e^{-n\sigma},
\end{equation*}
where $n=\min(n_s,n_\theta)$. Combining this estimate with the uniform stability bound for $\mathcal{A}_h^{-1}$ proves the theorem.
\end{proof}

Under the same analyticity and separation assumptions used above, the resulting multi-body discretization inherits the same exponential convergence behavior.

\begin{theorem}[Multi-body]
\label{thm:multi_conv}
Consider $N_{\mathrm{geo}}$ smooth quasi-axisymmetric scatterers, and assume that the separation between any two distinct scatterers is sufficiently large. Let $\boldsymbol{\mu}$ and $\boldsymbol{\mu}_h$ denote the solutions of the continuous and discrete multi-body double-layer boundary integral equations, respectively. Then
\begin{equation}
    \|\boldsymbol{\mu} - \boldsymbol{\mu}_h\|_\infty \leq C_M e^{-n\sigma},
\end{equation}
where $C_M > 0$ depends on the number and configuration of scatterers but the convergence rate $\sigma$ is the same as in the single-body case.
\end{theorem}
\begin{proof}
The multi-body operator can be written as $\mathcal{A}^M=\mathcal{A}_{\mathrm{diag}}+\mathcal{A}_{\mathrm{off}}$, where the diagonal part collects self-interactions and the off-diagonal part collects inter-body interactions. Under the well-separation condition, we can assume the separation distance is large enough such that $\|\mathcal{A}_{\mathrm{diag}}^{-1}\mathcal{A}_{\mathrm{off}}\|<1$. In this case, the operator $\mathcal{A}^M$ is invertible, and $\|(\mathcal{A}^M)^{-1}\|$ is uniformly bounded by the Neumann series. The diagonal blocks satisfy the same spectral convergence as in Theorem~\ref{thm:single_conv}, while the off-diagonal kernels are analytic on $\partial\Omega_i\times\partial\Omega_j$ for $i\neq j$ by the separation assumption. Hence their quadrature errors also decay like $Ce^{-n\sigma}$ by Theorem~\ref{thm:quad}. Combining these estimates yields the stated bound.
\end{proof}

\section{Numerical results}

\begin{table*}[!t] 
	\caption{Notations adopted in the subsequent tables}
	\label{tab: notations}
	\centering
	\begin{tabular}[b]
		{p{2.5cm}<{\raggedright}p{10cm}<{\raggedright}}
		\specialrule{0.1em}{4pt}{3pt}
		Notations  &Description\\
		\specialrule{0.05em}{3pt}{3pt}
		$k$  	 	 & Wavenumber\\
		$N_{p}$ & Number of panels to discretize the generating curve $\gamma$\\
		$N_{f}$  & Number of Fourier modes\\
        $N_{ref}$    & Number of dyadic refinements performed along the panel adjacent to the end point \\
		$N_{pts}$    & Total number of points to discretize the surface $\Gamma$\\
        $T_{ker}$    & Time (seconds) to evaluate all the modal functions \\
        $T_{mat}$    & Time (seconds) to construct the relevant matrix entries  \\
        $T_{solve}$  & Time (seconds) to solve the linear system  \\
		$Error$      & The $\ell^2$ error of the numerical solutions\\
		\specialrule{0.1em}{2pt}{0pt}
	\end{tabular}
\end{table*}

This section presents a series of numerical experiments designed to assess the accuracy and efficiency of the proposed scheme for a range of quasi-axisymmetric scattering objects. We consider both exterior and interior scattering problems, as well as single- and multiple-scatterer configurations, over several wavenumbers. We test results for both the single-layer and double-layer formulations.

Because analytic reference solutions are generally unavailable for these geometries, we estimate the numerical error by means of artificial solutions. More precisely, the reference field is generated by $L$ point sources $\mathrm{\Phi}(\mathbf{x},\mathbf{y}_l)$ with source points $\mathbf{y}_l$ placed inside or outside $\Omega$:
\begin{equation} \label{artificial point source solution}
    u(\mathbf{x})=\sum_{l=1}^{L}\frac{e^{ik|\mathbf{x}-\mathbf{y}_l|}}{4\pi|\mathbf{x}-\mathbf{y}_l|}.
\end{equation}
After computation, the $\ell^2$ error is evaluated over a large set of randomly selected test points by comparing the numerical solution against the reference solution.



We also consider  the physical scattering problem for which no reference solution is available. Specifically, we solve the integral equation with an incident plane wave
\begin{equation}
    u^{inc} = A e^{ik \mathbf{d} \cdot \mathbf{x}},
\end{equation}
where $A=1$ is the amplitude and the unit vector $\mathbf{d}$ specifies the propagation direction,
\begin{equation} \label{plane wave parameters}
\mathbf{d}=(\cos\theta\sin\varphi, \sin\theta\sin\varphi, \cos\varphi),
\end{equation}
with $\theta=\pi/3, \varphi=2\pi/3$.
The asymptotic behavior of the scattered field at infinity is given by
\begin{equation}
\begin{aligned}
    u(x) 
    &= \frac{e^{ik|\mathbf{x}|}}{|\mathbf{x}|}u_{\infty}(\theta, \varphi)+ \mathcal{O}\left(\frac{1}{|\mathbf{x}|^2} \right),\quad |\mathbf{x}|\to+\infty.
\end{aligned}
\end{equation}
Here $u_\infty$ is referred to as the far-field pattern
\begin{equation}
    u_{\infty}(\theta, \varphi) = \frac{1}{4\pi} \int_{\partial\Omega} e^{-ik\mathbf{y} \cdot \frac{\mathbf{x}}{|\mathbf{x}|}} \mu (\mathbf{y}) ds(\mathbf{y}) ,
\end{equation}
where $\mu$ is the single-layer density obtained by solving the boundary integral equation. Self-convergence tests are performed on $u_{\infty}(\theta,\varphi)$ using 200 equispaced sampling points in each of the azimuthal and polar directions on the unit sphere.

For a single target, matrix assembly is dominated by the evaluation of the modal Green's functions and the application of the modified quadrature rule. In all single-body experiments, we apply a 16th-order Nyström-like discretization to each integral equation along the generating curve. The tolerance used for kernel evaluation and for truncating the Fourier series in the discrete convolutions is set to $10^{-12}$.

In the multi-body setting, the modal Green's functions associated with self-interaction need to be evaluated only once because all target bodies are identical. Additional computational cost arises from the mutual interactions between distinct bodies. Under mesh refinement, the matrix size grows rapidly, and the cost of solving the resulting dense linear systems increases accordingly.
In all experiments except the large-scale regular array examples in Section~\ref{sec:cubic}, we use MATLAB's direct solver (\texttt{mldivide}) to solve the dense linear system.
For the large-scale cubic arrays in Section~\ref{sec:cubic}, we adopt a hybrid solver strategy: GMRES preconditioned by a block-diagonal preconditioner is used for low wavenumber, while the direct solver is
used for high wavenumber.
The block-diagonal preconditioner is constructed from the LU factorization of the single-body self-interaction matrix, which is computed once and shared by all scatterers due to the identical geometry.
The switch to direct solve at high wavenumber is motivated by the stronger inter-body coupling at higher wavenumbers, which slows GMRES convergence and makes the direct solver more competitive.
While MATLAB's direct solvers are efficient for this work's scale, the prohibitive $\mathcal{O}(N^3)$ complexity for millions of unknowns would necessitate fast algorithms like the Fast Multipole Method(FMM) \cite{article60} or preconditioned iterative solvers.

All experiments were implemented in MATLAB and carried out on a server equipped with an Intel Xeon CPU and 256 GB of RAM. The notation used in the tables is summarized in Table~\ref{tab: notations}.

\subsection{Single-body scattering}

\subsubsection{Example 1: Scattering from a quasi-wave geometry}
\label{sec:single}

\begin{figure}[!t]
\centering

\begin{subfigure}[t]{0.3\textwidth}
\centering
\includegraphics[width=1.4in]{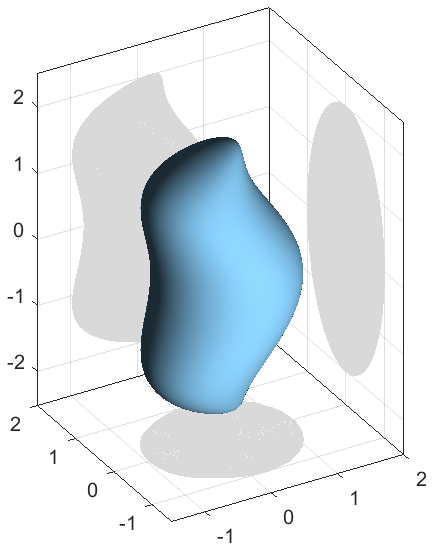}
\caption{quasi-wave geometry.}
\label{geo5}
\end{subfigure}
\hfill
\begin{subfigure}[t]{0.34\textwidth}
\centering
\includegraphics[width=2.1in]{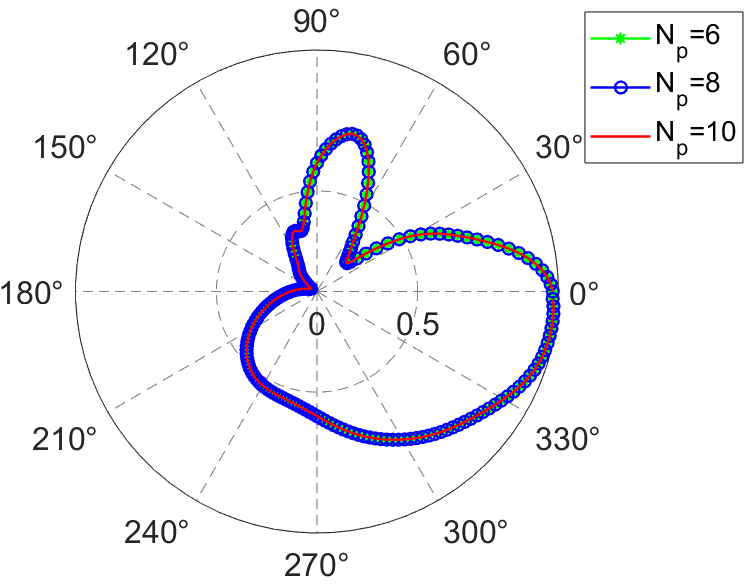}
\caption{$|u_\infty(\cdot,\pi/2)|.$}
\label{subfig: quasiwave far theta}
\end{subfigure}
\hfill
\begin{subfigure}[t]{0.34\textwidth}
\centering
\includegraphics[width=2.1in]{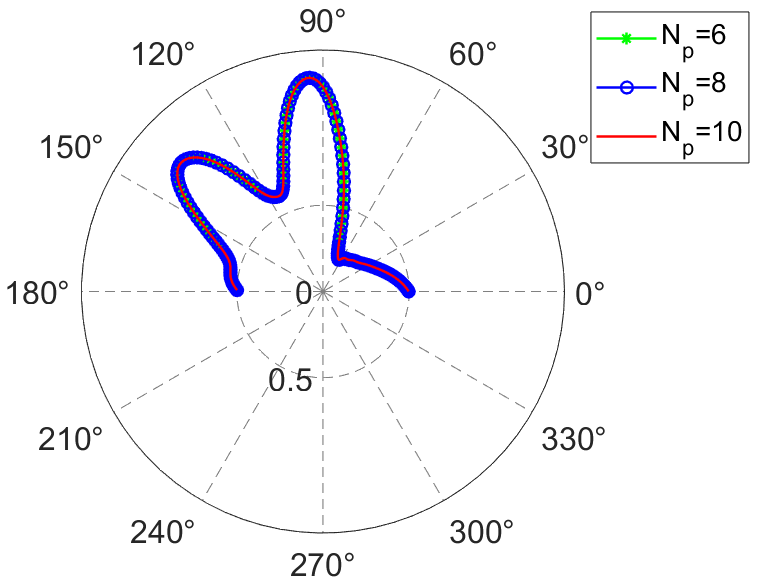}
\caption{$|u_\infty(0,\cdot)|.$}
\label{subfig: quasiwave far phi}
\end{subfigure}

\caption{Numerical results for the quasi-wave geometry with $k=10$ and $N_f=40$.}
\label{fig: quasiwave}
\end{figure}
\begin{table*}[!t]
    \caption{Exterior scattering from the quasi-wave geometry using the single-layer potential}
    \label{tab:geo2singleexterior}
    \centering
    \renewcommand{\arraystretch}{0.8}
    \begin{tabular}
    {p{0.5cm}<{\raggedright}p{1cm}<{\raggedright}p{1cm}<{\raggedright}p{1cm}<{\raggedright}p{1cm}
    <{\raggedright}p{1.5cm}<{\raggedright}p{1.5cm}
    <{\raggedright}p{1.5cm}
    <{\raggedright}p{2cm}
    <{\raggedright}}
    \specialrule{0.1em}{4pt}{4pt}
    $k$&  $N_{ref}$&  $N_{p}$&  $N_{f}$&  $N_{pts}$&  $T_{ker}$&  $T_{mat}$&  $T_{solve}$&  $Error$ \\
    \specialrule{0.05em}{4pt}{4pt}
    2&	0&    6&	40&	3840&	75.64 & 	81.12 & 	1.49& 	3.21E-13\\
    2&	2&   16&	40&	10240&	173.35&	    186.06&	    6.11&	4.40E-16\\
    2&	5&   22&	60& 21120&  632.00&	    674.50&	   45.06&	4.01E-17\\
    \specialrule{0em}{3pt}{3pt}
    5&	0&    6&	40&	3840&	46.58 & 	50.12 & 	1.50& 	1.04E-11\\
    5&	2&    16&	40&	10240&	278.19&	    298.51&	   22.39&	2.29E-14\\
    5&	5&    22&	60& 21120&  531.88&     571.86&	   46.95&	1.29E-16\\
    \specialrule{0em}{3pt}{3pt}
    10&	0&    8&	60&	7680&	103.82& 	117.74& 	3.18& 	4.27E-11\\
    10&	2&   18&	60&	17280&	327.84&	    354.75&    37.05&	5.41E-14\\
    10&	5&   24&	60&	23040&	776.72&	    832.36&    77.47&	3.46E-15\\
    \specialrule{0.1em}{2pt}{0pt}
    \end{tabular}
\end{table*}

\begin{table*}[!t]
    \caption{Exterior scattering from the quasi-wave geometry using the double-layer potential} 
    \label{tab:geo2doubleexterior}
    \centering
    \renewcommand{\arraystretch}{0.8}
    \begin{tabular}
    {p{0.5cm}<{\raggedright}p{1cm}<{\raggedright}p{1cm}<{\raggedright}p{1cm}<{\raggedright}p{1cm}
    <{\raggedright}p{1.5cm}<{\raggedright}p{1.5cm}
    <{\raggedright}p{1.5cm}
    <{\raggedright}p{2cm}
    <{\raggedright}}
    \specialrule{0.1em}{4pt}{4pt}
    $k$&  $N_{ref}$&  $N_{p}$&  $N_{f}$&  $N_{pts}$&  $T_{ker}$&  $T_{mat}$&  $T_{solve}$&  $Error$ \\
    \specialrule{0.05em}{4pt}{4pt}
    2&	0&    6&	40&	3840&	76.98 & 	83.39 & 	1.29& 	1.09E-11\\
    2&	2&    16&	40&	10240&	257.16&	    281.12&	   8.43 &	8.79E-15\\
    2&	5&    22&	60& 21120&  964.67&    1052.73&	   56.13&	1.14E-15\\
    \specialrule{0em}{3pt}{3pt}
    5&	0&    6&	40&	3840&	70.61 & 	77.89 & 	0.98& 	1.60E-11\\
    5&	2&    16&	40&	10240&	277.45&	    301.17&	   9.03 &	1.86E-14\\
    5&	5&    22&	60& 21120&  874.43&     950.68&	   49.35&	3.76E-16\\
    \specialrule{0em}{3pt}{3pt}
    10&	0&    8&	60&	7680&	148.29& 	160.06& 	3.64& 	5.15E-11\\
    10&	2&   18&	60&	17280&	555.05&	    603.71&    27.35&	4.16E-15\\
    10&	5&   24&	60&	23040& 1022.15&	   1103.79&    58.37&	2.71E-15\\
    \specialrule{0.1em}{2pt}{0pt}
    \end{tabular}
\end{table*}

To illustrate the applicability of the method to periodically undulating surfaces, we first consider a quasi-wave geometry parameterized by
\begin{equation}
\label{eq:quasiwave}
\text{Quasi-wave:}\qquad
\begin{cases}
x = 0.2\cos(4t)+\cos(t)\cos(\varphi), \\
y = \cos(t)\sin(\varphi), \\
z = 2\sin(t),
\end{cases}
\end{equation}
where $t\in[-\pi/2,\pi/2]$ and $\varphi\in[0,2\pi]$. See Figure~\ref{geo5} for an illustration of the geometry. 
Tables~\ref{tab:geo2singleexterior} and~\ref{tab:geo2doubleexterior} show the results obtained with the single-layer and double-layer formulations, respectively.

The method achieves high accuracy for all tested wavenumbers, with the error decreasing as the number of panels and Fourier modes increases. As expected, higher wavenumbers require finer discretizations to maintain the same accuracy. Dyadic refinement near the poles further improves the results across all cases. For instance, when $k=10$, using five levels of dyadic refinement reduces the error to about $10^{-15}$, compared with about $10^{-11}$ for the unrefined computation with eight panels. This  improvement shows that the proposed discretization remains effective for the quasi-wave geometry. The matrix assembly time dominates the total cost, particularly for the double-layer formulation, whose kernel evaluation is more involved.

The corresponding far field patterns are shown in Figures~\ref{subfig: quasiwave far theta} and~\ref{subfig: quasiwave far phi}. With the number of Fourier modes fixed at $N_f=40$, the far field pattern converges rapidly as the number of panels increases. In the $xy$ plane, $\phi=\pi/2$, the pattern contains several pronounced lobes and deep nulls, indicating strong interference effects caused by the periodic surface structure. In contrast, in the $xz$ plane $\theta=0$, the scattering is dominated by a main forward lobe with weaker side lobes distributed symmetrically on both sides. This simpler structure occurs because the quasi-wave profile is viewed along its axial direction, producing a less intricate interference pattern than in the azimuthal plane.

\subsubsection{Example 2: Spectral convergence for quasi-axisymmetric geometries}
\begin{figure}[!t]
    \begin{minipage}[b]{0.32\textwidth}
        \centering
        \includegraphics[height=1.6in]{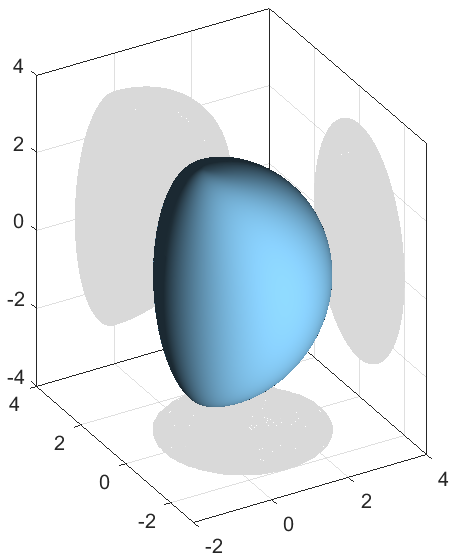}
        \subcaption{Quasi-ellipsoid.}
        \label{geo1}
    \end{minipage}
    \begin{minipage}[b]{0.32\textwidth}
        \centering
        \includegraphics[height=1.6in]{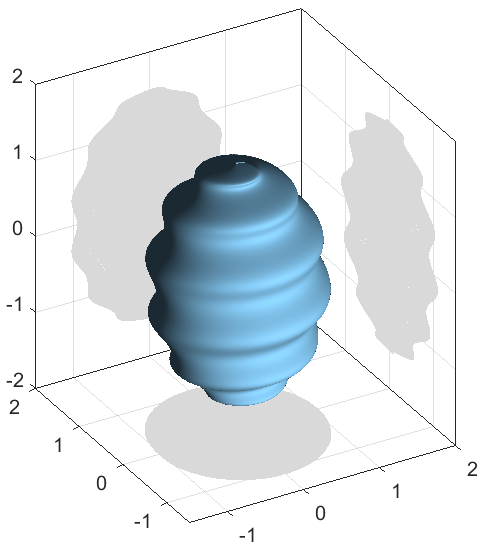}
        \subcaption{Spiral.}
        \label{spiral}
    \end{minipage}
    \begin{minipage}[b]{0.32\textwidth}
        \centering
        \includegraphics[height=1.6in]{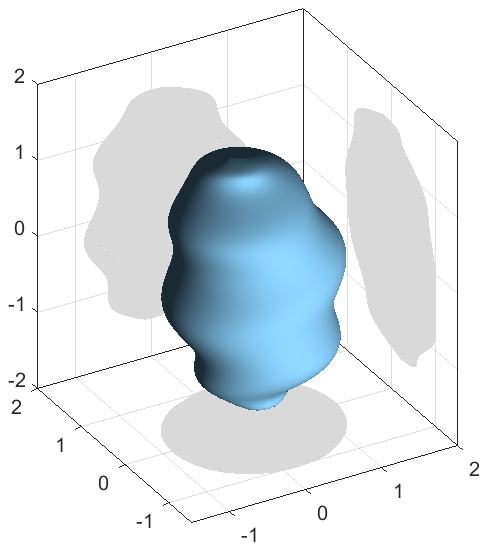}
        \subcaption{Conch.}
        \label{conch}
    \end{minipage}

    \vspace{0.2cm}

    \begin{minipage}[b]{0.32\textwidth}
        \centering
        \includegraphics[height=1.5in]{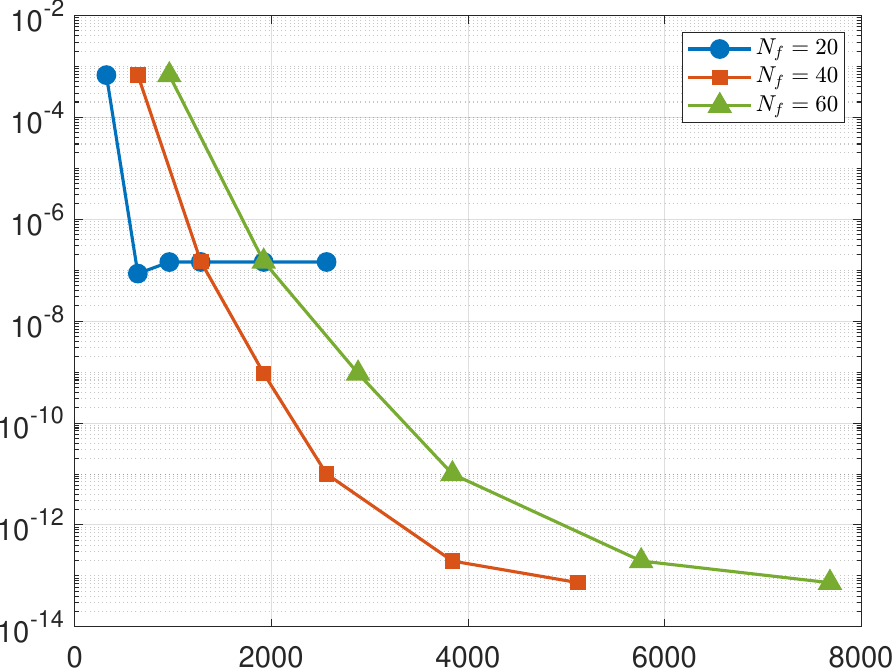}
        \subcaption{Convergence for quasi-ellipsoid.}
        \label{conv_geo1}
    \end{minipage}
    \begin{minipage}[b]{0.32\textwidth}
        \centering
        \includegraphics[height=1.5in]{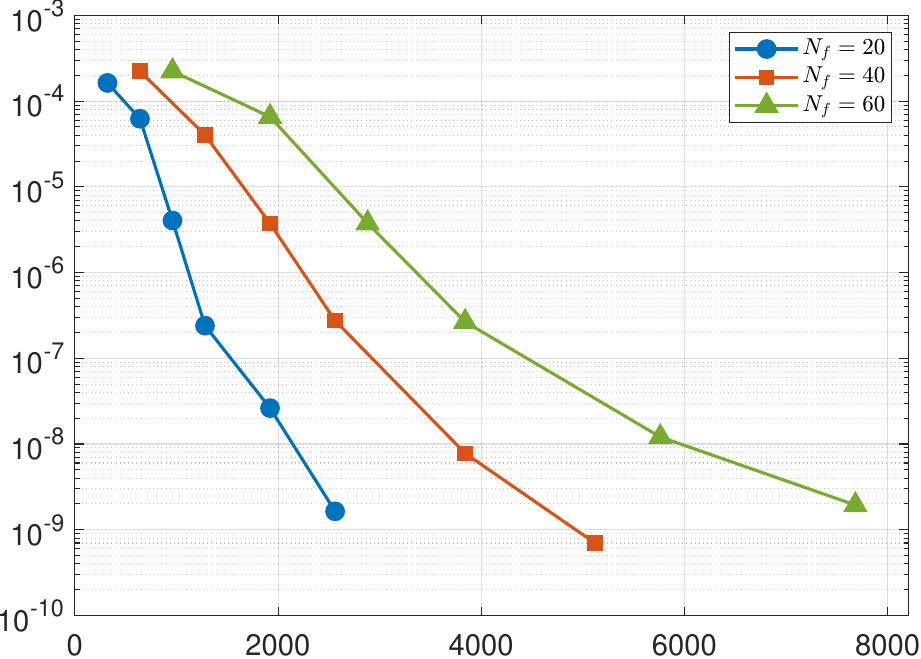}
        \subcaption{Convergence for spiral.}
        \label{conv_spiral}
    \end{minipage}
    \begin{minipage}[b]{0.32\textwidth}
        \centering
        \includegraphics[height=1.5in]{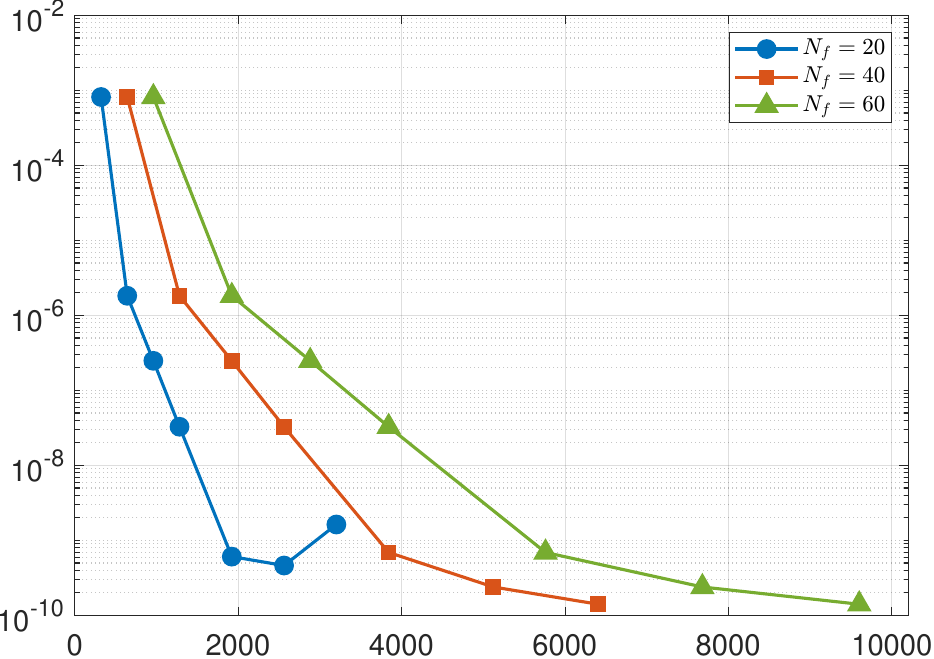}
        \subcaption{Convergence for conch.}
        \label{conv_conch}
    \end{minipage}

    \caption{Three quasi-axisymmetric geometries and the corresponding spectral convergence for $k=5$.}
    \label{fig:geometries}
\end{figure}

\begin{table*}[!t]
    \caption{Accuracy and efficiency for the quasi-axisymmetric ellipsoid, spiral, and conch geometries with $k=5$ and dyadic refinement near the poles.}
    \label{tab:single_dyadic}
    \centering
    \renewcommand{\arraystretch}{0.8}
    \begin{tabular}
    {p{2cm}<{\raggedright}p{0.8cm}<{\raggedright}p{0.8cm}<{\raggedright}p{0.8cm}
    <{\raggedright}p{1cm}<{\raggedright}p{1.5cm}<{\raggedright}p{1.5cm}
    <{\raggedright}p{1.5cm}<{\raggedright}p{2cm}<{\raggedright}}
    \specialrule{0.1em}{4pt}{4pt}
    Geometry & $N_p$ & $N_{\mathrm{ref}}$ & $N_f$ & $N_{\mathrm{pts}}$ & $T_{\mathrm{ker}}$ & $T_{\mathrm{mat}}$ & $T_{\mathrm{solve}}$ & $\mathrm{Error}$ \\
    \specialrule{0.05em}{4pt}{4pt}
    \multirow{3}{*}{\shortstack[l]{Quasi-\\ellipsoid}}
      & 6  & 0 & 40 & 3840  & 63.49  & 67.69  & 0.60  & 4.07E-11 \\
      & 16 & 2 & 60 & 15360 & 423.98 & 458.44 & 19.70 & 8.08E-14 \\
      & 22 & 5 & 60 & 21120 & 744.32 & 809.59 & 43.82 & 4.55E-15 \\
    \specialrule{0em}{3pt}{3pt}
    \multirow{3}{*}{Spiral}
      & 8  & 0 & 40 & 5120  & 112.71 & 119.90 & 1.76  & 1.92E-08 \\
      & 16 & 2 & 40 & 10240 & 230.72 & 248.82 & 6.96  & 2.01E-11 \\
      & 22 & 5 & 60 & 21120 & 878.91 & 945.30 & 45.54 & 1.93E-13 \\
    \specialrule{0em}{3pt}{3pt}
    \multirow{3}{*}{Conch}
      & 6  & 0 & 40 & 3840  & 68.21  & 72.52  & 0.82  & 2.53E-09 \\
      & 16 & 2 & 40 & 10240 & 258.97 & 277.75 & 7.83  & 2.01E-11 \\
      & 22 & 5 & 60 & 21120 & 864.51 & 930.71 & 42.56 & 1.93E-13 \\
    \specialrule{0.1em}{2pt}{0pt}
    \end{tabular}
\end{table*}

Figure~\ref{fig:geometries} shows three representative quasi-axisymmetric geometries: a quasi-ellipsoid, a spiral, and a conch, together with their convergence behavior for $k=5$. 
The corresponding parametrization equations are given by
\begin{subequations}\label{eq:geometries}
\begin{align}
\text{Quasi-ellipsoid:}\quad &
\begin{cases}
x = \cos(t)+2\cos(t)\cos(\varphi), \\
y = 2\cos(t)\sin(\varphi), \\
z = 3\sin(t),
\end{cases}
\label{eq:quasiellipsoid}
\\
\text{Spiral:}\quad &
\begin{cases}
x = 0.05\cos(4\pi t)+1.5\cos(t)\cos(\varphi),\\
y = 0.1\sin(4\pi t)+1.5\cos(t)\sin(\varphi),\\
z = 2\sin(t),
\end{cases}
\label{eq:spiral}
\\
\text{Conch:}\quad &
\begin{cases}
x = \left(0.2+0.1\sin(8t)\right)\cos(t)+\cos(t)\cos(\varphi),\\
y = \left(0.2+0.1\sin(8t)\right)\sin(t)+\cos(t)\sin(\varphi),\\
z = 1.5\sin(t),
\end{cases}
\label{eq:conch}
\end{align}
\end{subequations}
where \(t\in[-\pi/2,\pi/2]\) and \(\varphi\in[0,2\pi]\). 
For each geometry, the $\ell^2$ error of the scattered field is plotted against the total number of discretization points $N_{\mathrm{tot}}$ for three choices of the azimuthal mode number $N_f$.

The error decays rapidly as $N_{\mathrm{tot}}$ increases for all three geometries, confirming the spectral accuracy of the proposed method across different shapes. Consistent with the analysis in Section~\ref{sec:error}, accurate resolution requires balanced refinement in the $s$- and $\theta$-directions. If $N_f$ is too small, the error eventually saturates at a level determined by the azimuthal discretization, and further increasing $N_p$ alone does not improve the accuracy. Increasing $N_f$ lowers this saturation level, so both discretization directions must be refined together to obtain high precision.

For geometries with large curvature variation near the poles, such as the spiral and conch, uniform refinement eventually stagnates because of endpoint singularities in the generating curve. Table~\ref{tab:single_dyadic} shows that combining uniform refinement with dyadic refinement near the poles restores high accuracy. With five dyadic refinement levels ($N_{\mathrm{ref}}=5$), the errors are reduced to between \(10^{-15}\) and \(10^{-13}\), representing a three to four orders of magnitude improvement over uniform refinement alone.

\subsection{Multi-body scattering}

\subsubsection{Example 3: Spectral convergence for multiple scatterers}

\begin{figure}[!t]
    \begin{minipage}[t]{0.32\textwidth}
        \centering
        \includegraphics[width=1.7in]{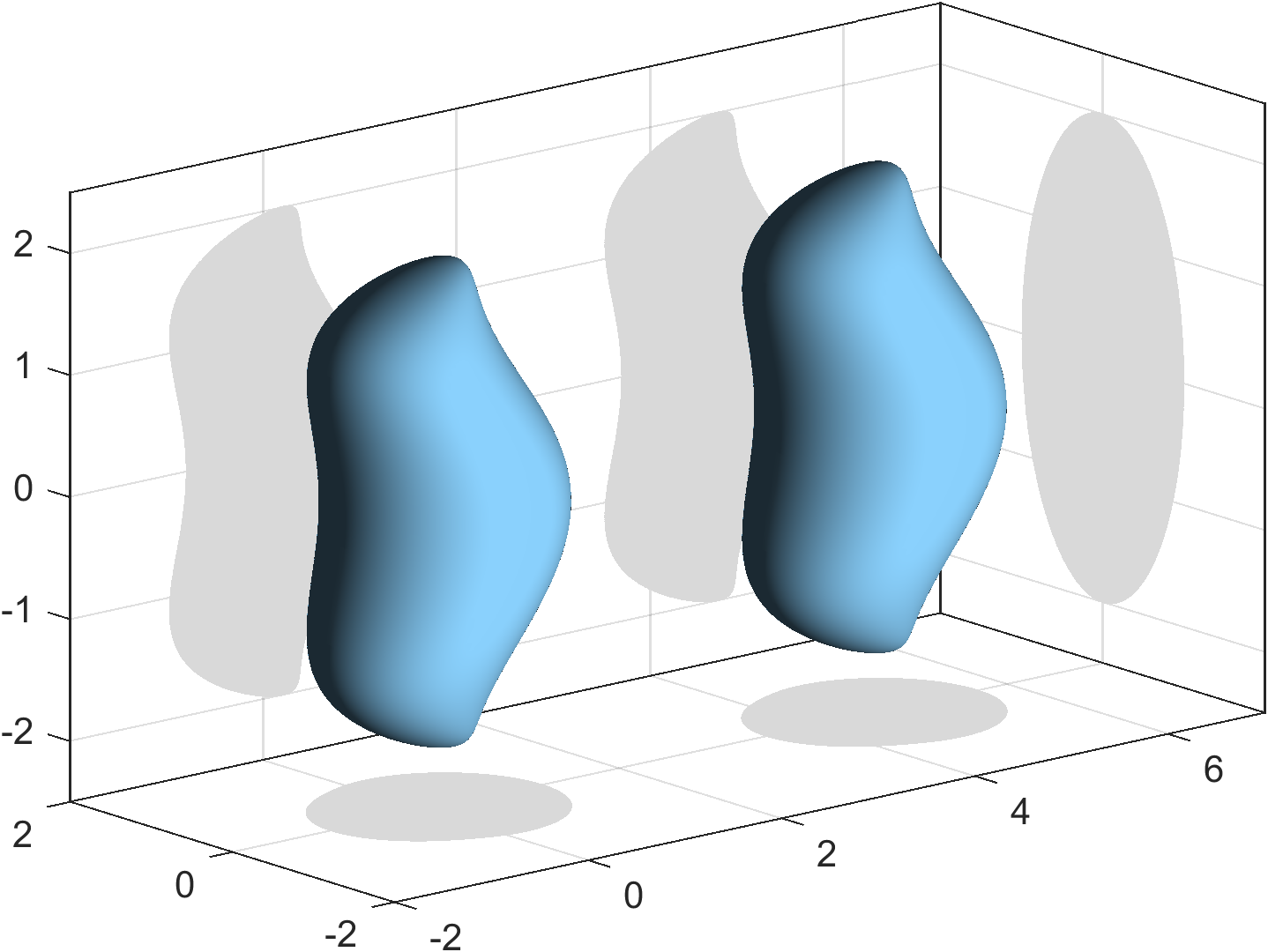}
        \subcaption{Two quasi-wave geometries.}
        \label{fig:twospiral}
    \end{minipage}
    \begin{minipage}[t]{0.32\textwidth}
        \centering
        \includegraphics[width=1.8in]{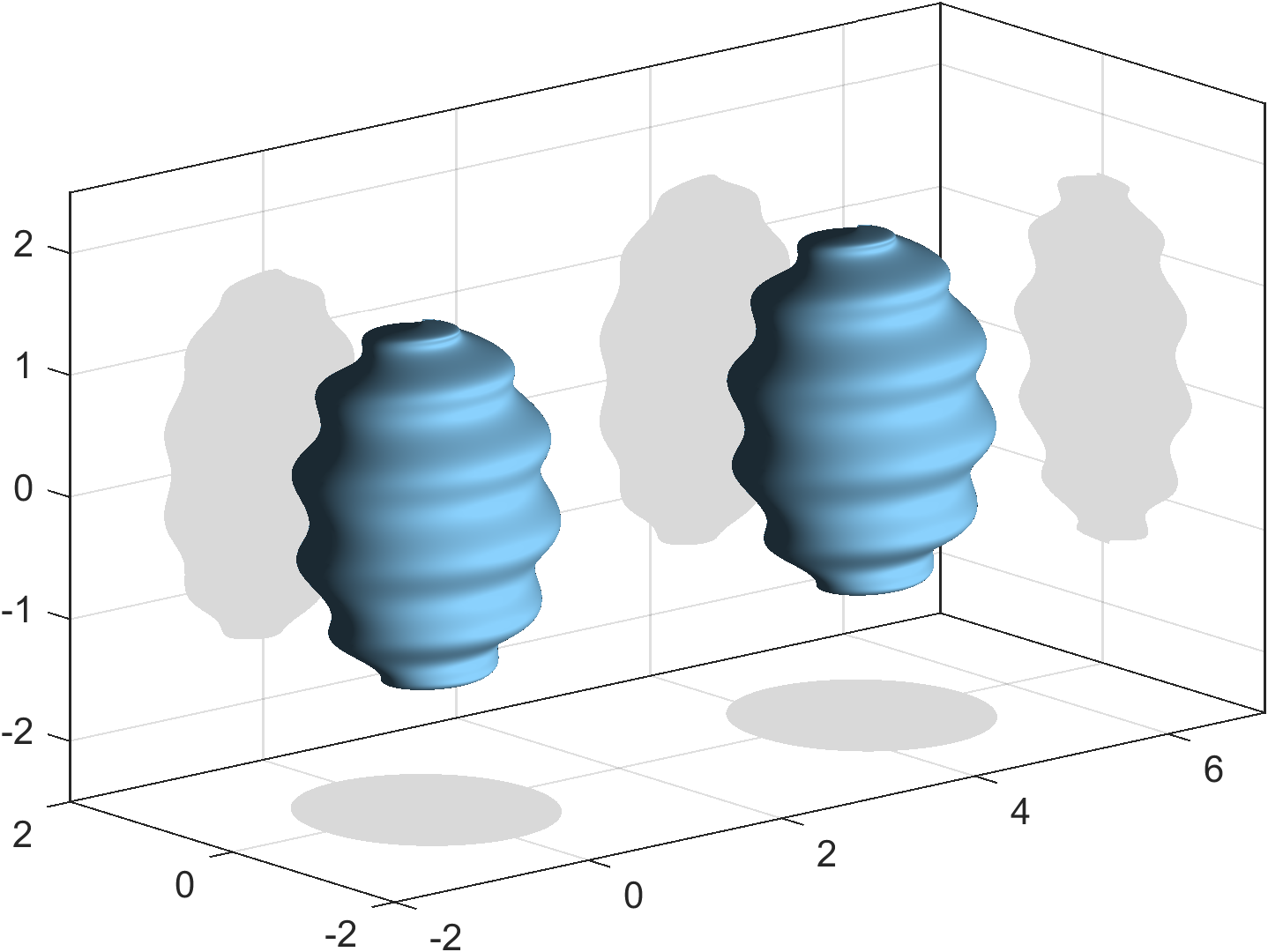}
        \subcaption{Two spiral geometries.}
        
    \end{minipage}
    \begin{minipage}[t]{0.32\textwidth}
        \centering
        \includegraphics[width=1.8in]{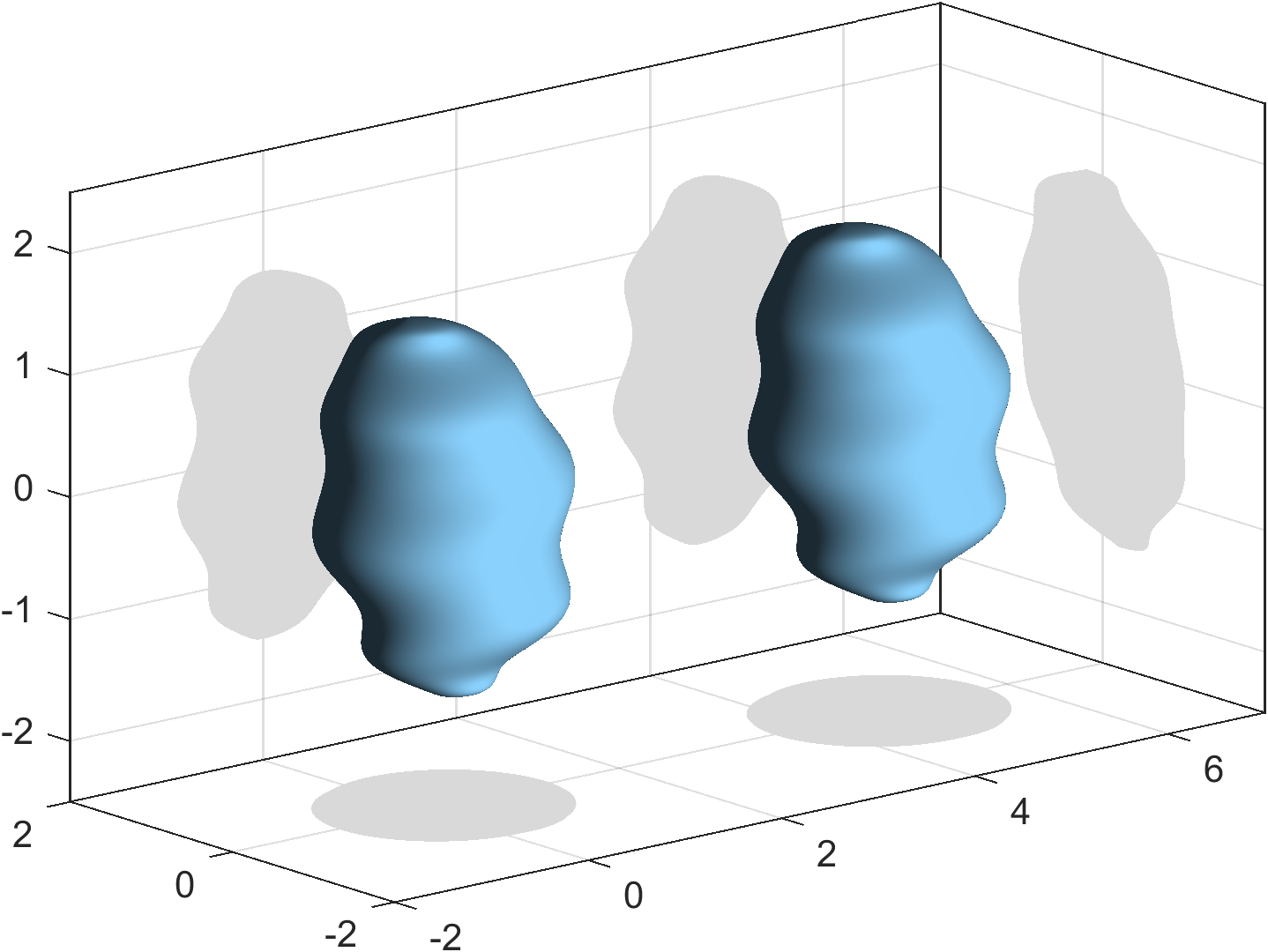}
        \subcaption{Two conch geometries.}
        
    \end{minipage}
    \caption{Two-body scattering configurations used in the convergence tests.}
    \label{fig:twobody}
\end{figure}

\begin{table*}[!t]
    \caption{Exterior scattering by two scatterers ($N_{\mathrm{geo}}=2$) for three representative geometries.}
    \label{tab:multibody2}
    \centering
    \renewcommand{\arraystretch}{0.8}
    \begin{subtable}[t]{0.32\textwidth}
        \centering
        \caption{Quasi-wave.}
        \label{tab:multibody2_quasiwave}
        \begin{tabular}
        {p{0.2cm}<{\raggedright}p{0.2cm}<{\raggedright}p{0.4cm}<{\raggedright}
        p{0.8cm}<{\raggedright}p{1.6cm}<{\raggedright}}
        \specialrule{0.1em}{4pt}{4pt}
        $k$ & $N_p$ & $N_f$ & $N_{\mathrm{pts}}$ & $\mathrm{Error}$ \\
        \specialrule{0.05em}{4pt}{4pt}
        2  & 2 & 20 & 640  & 4.48E-09 \\
        2  & 4 & 20 & 1280 & 1.10E-11 \\
        2  & 6 & 20 & 1920 & 6.83E-13 \\
        \specialrule{0em}{3pt}{3pt}
        5  & 4 & 40 & 2560 & 1.49E-10 \\
        5  & 6 & 40 & 3840 & 1.51E-12 \\
        5  & 8 & 40 & 5120 & 1.09E-13 \\
        \specialrule{0em}{3pt}{3pt}
        10 & 6 & 40 & 3840 & 9.18E-10 \\
        10 & 8 & 40 & 5120 & 3.66E-12 \\
        \specialrule{0.1em}{2pt}{0pt}
        \end{tabular}
    \end{subtable}
    \hfill
    \begin{subtable}[t]{0.32\textwidth}
        \centering
        \caption{Spiral.}
        \label{tab:multibody2_spiral}
        \begin{tabular}
        {p{0.2cm}<{\raggedright}p{0.2cm}<{\raggedright}p{0.4cm}<{\raggedright}
        p{0.8cm}<{\raggedright}p{1.6cm}<{\raggedright}}
        \specialrule{0.1em}{4pt}{4pt}
        $k$ & $N_p$ & $N_f$ & $N_{\mathrm{pts}}$ & $\mathrm{Error}$ \\
        \specialrule{0.05em}{4pt}{4pt}
        2  & 2 & 40 & 1280 & 3.61E-06 \\
        2  & 4 & 40 & 2560 & 4.99E-07 \\
        2  & 6 & 40 & 3840 & 6.31E-08 \\
        \specialrule{0em}{3pt}{3pt}
        5  & 4 & 40 & 2560 & 1.62E-07 \\
        5  & 6 & 40 & 3840 & 3.61E-09 \\
        5  & 8 & 40 & 5120 & 4.49E-10 \\
        \specialrule{0em}{3pt}{3pt}
        10 & 6 & 60 & 5760 & 1.11E-06 \\
        10 & 8 & 60 & 7680 & 5.27E-07 \\
        \specialrule{0.1em}{2pt}{0pt}
        \end{tabular}
    \end{subtable}
    \hfill
    \begin{subtable}[t]{0.32\textwidth}
        \centering
        \caption{Conch.}
        \label{tab:multibody2_conch}
        \begin{tabular}
        {p{0.2cm}<{\raggedright}p{0.2cm}<{\raggedright}p{0.4cm}<{\raggedright}
        p{0.8cm}<{\raggedright}p{1.6cm}<{\raggedright}}
        \specialrule{0.1em}{4pt}{4pt}
        $k$ & $N_p$ & $N_f$ & $N_{\mathrm{pts}}$ & $\mathrm{Error}$ \\
        \specialrule{0.05em}{4pt}{4pt}
        2  & 2 & 40 & 1280 & 3.32E-07 \\
        2  & 4 & 40 & 2560 & 7.29E-10 \\
        2  & 6 & 40 & 3840 & 6.49E-11 \\
        \specialrule{0em}{3pt}{3pt}
        5  & 4 & 40 & 2560 & 2.29E-08 \\
        5  & 6 & 40 & 3840 & 2.36E-10 \\
        5  & 8 & 40 & 5120 & 1.30E-10 \\
        \specialrule{0em}{3pt}{3pt}
        10 & 6 & 60 & 5760 & 9.88E-09 \\
        10 & 8 & 60 & 7680 & 2.83E-10 \\
        \specialrule{0.1em}{2pt}{0pt}
        \end{tabular}
    \end{subtable}
\end{table*}

We first consider two-scatterer configurations consisting of quasi-wave, spiral, and conch geometries, as shown in Figure~\ref{fig:twobody}. 
Table~\ref{tab:multibody2} gives the corresponding exterior scattering results for $k=2,5,10$. For each geometry and wavenumber, the error decreases as $N_p$ increases, confirming that the multi-body solver preserves the rapid convergence observed for the single-body discretization.

The quasi-wave geometry reaches high accuracy with relatively few panels: the error is $\mathcal{O}(10^{-13})$ for $k=2$ with $N_p=6$, $N_f=20$, and also for $k=5$ with $N_p=8$, $N_f=40$. The spiral and conch geometries are more demanding because of their more complex azimuthal structure, and they therefore require larger values of $N_f$ to reach comparable accuracy. At $k=10$, the spiral case has the largest errors at the same discretization level, reflecting the combined difficulty of higher frequency and more complicated geometry. Even so, the conch case still achieves $\mathcal{O}(10^{-10})$ accuracy with $N_p=8$ and $N_f=60$. Overall, these two-body experiments are consistent with the single-body results in Section~\ref{sec:single}. They also support Theorem~\ref{thm:multi_conv}, namely, the inter-body coupling does not degrade the spectral accuracy of the underlying Nystr\"om discretization when the scatterers are well separated.

\begin{figure}[!t]
    \begin{minipage}[t]{0.32\textwidth}
        \centering
        \includegraphics[width=1.8in]{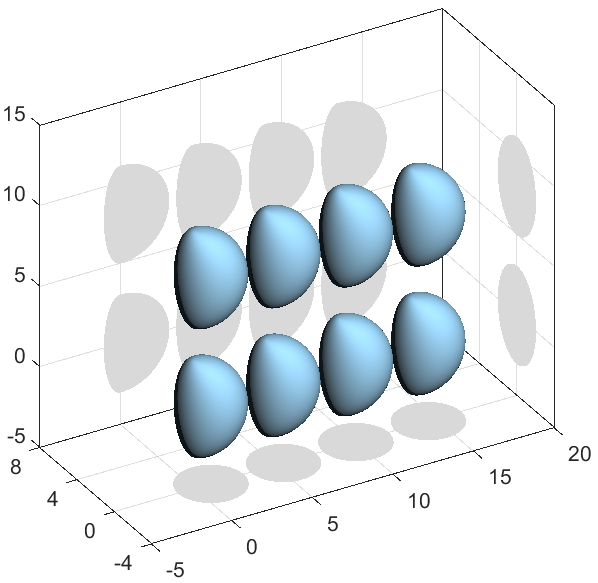}
\subcaption{Eight quasi-axisymmetric ellipsoids.}
        \label{subfig: blue 8}
    \end{minipage}
    \begin{minipage}[t]{0.33\textwidth}
	\centering
	\includegraphics[width=2in]{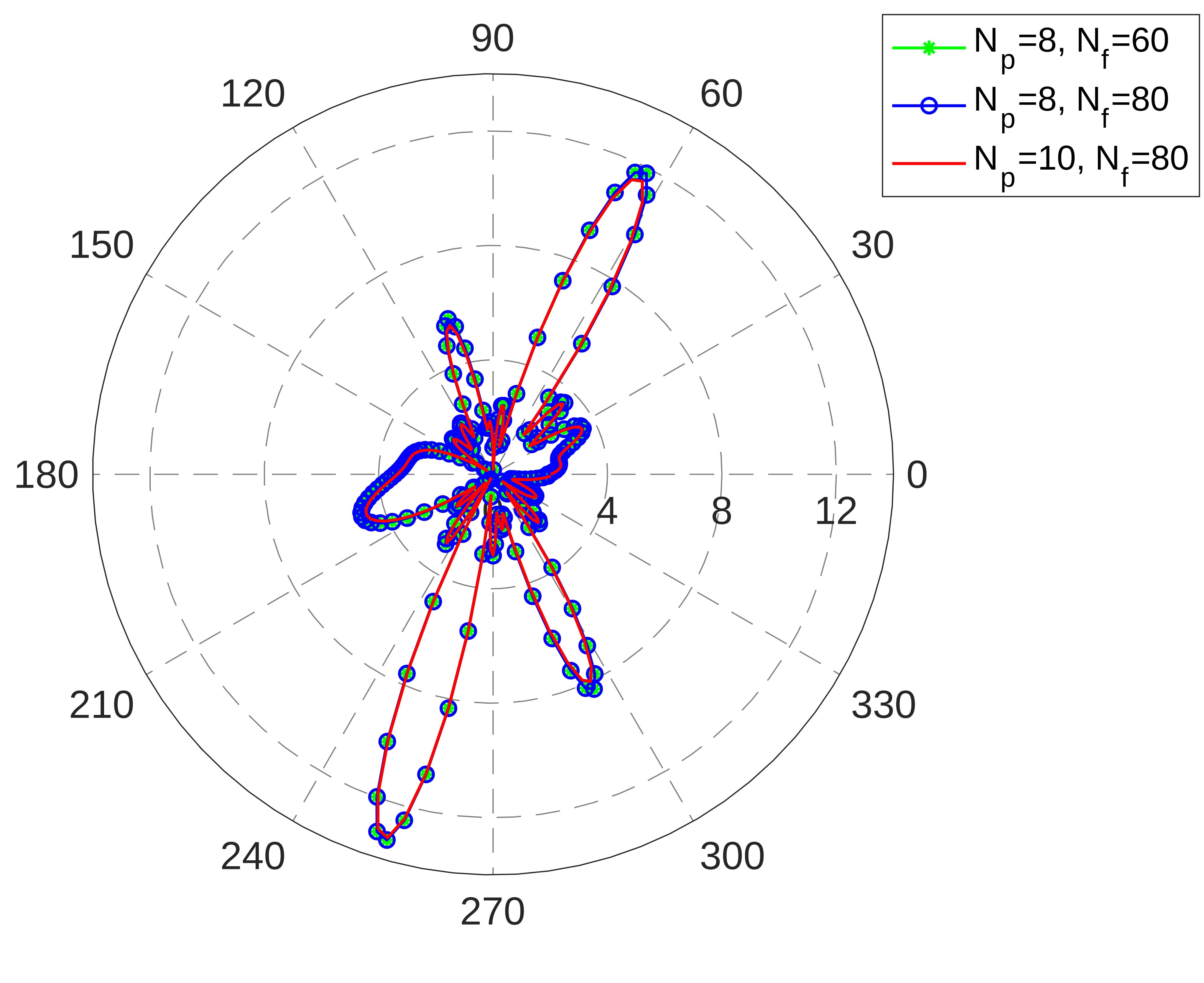}
	\subcaption{$|u_\infty(\cdot,\pi/2)|.$}
	\label{subfig: 8 far theta}
    \end{minipage}
    \begin{minipage}[t]{0.33\textwidth}
	\centering
	\includegraphics[width=2in]{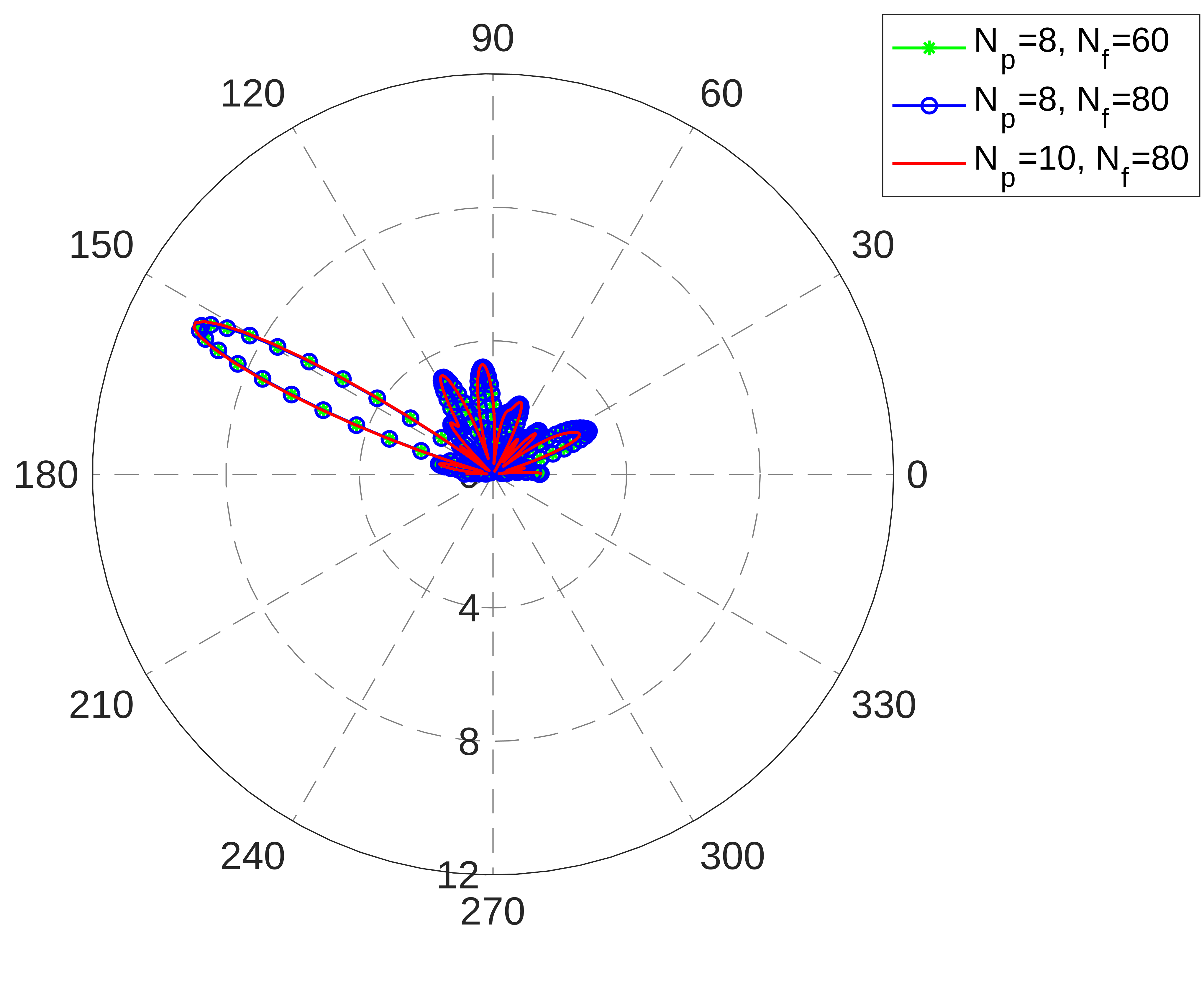}
	\subcaption{$|u_\infty(0,\cdot)|.$}
	\label{subfig: 8 far phi}
    \end{minipage}
\caption{Far-field results for eight quasi-axisymmetric ellipsoids with $k=10$.}
\label{fig: farfield8}
\end{figure}

\begin{table*}[!t]
    \caption{Exterior scattering by eight quasi-axisymmetric ellipsoids using the double-layer potential}
    \label{tab:geo1doubleexterior8}
    \centering
    \renewcommand{\arraystretch}{0.8}
    \begin{tabular}
    {p{0.5cm}<{\raggedright}p{1cm}<{\raggedright}p{1cm}<{\raggedright}p{1cm}
    <{\raggedright}p{1.5cm}<{\raggedright}p{1.5cm}
    <{\raggedright}p{1.5cm}
    <{\raggedright}p{2cm}
    <{\raggedright}}
    \specialrule{0.1em}{4pt}{4pt}
    $k$&  $N_{p}$&  $N_{f}$&  $N_{pts}$&  $T_{ker}$&  $T_{mat}$&  $T_{solve}$&  $Error$ \\
    \specialrule{0.05em}{4pt}{4pt}
    2&	6 &	40&	3840 &	65.85 &	 138.65  &  89.40  &	2.68E-10\\
    2&	6 &	60&	5760 &	105.62&	 264.19  &  258.48 &	2.68E-12\\
    2&	10&	80&	12800 &	281.78&	 1078.61 &  2486.45&	4.83E-13\\
    \specialrule{0em}{3pt}{3pt}
    5&	6 &	40&	3840 &	65.46 &	 138.48 &  99.89   &	3.75E-09\\
    5&	6 &	60&	5760 &	103.57&	 262.48 &  296.11  &	3.78E-11\\
    5&	10&	80&	12800 &	273.63&	 1076.73&  2575.58 &	5.41E-12\\
    \specialrule{0em}{3pt}{3pt}
    10&	6 &	60&	5760 &	104.93&  262.45 &  246.81  &	2.38E-07\\
    10&	8 &	60&	7680 &	151.72&  421.06 &  619.35  &	3.37E-09\\
    10&	10&	80&	12800 &	277.35&	 1079.94&  2786.19 &	3.52E-11\\
    \specialrule{0.1em}{2pt}{0pt}
    \end{tabular}
\end{table*}

We next consider a larger configuration consisting of eight translated quasi-axisymmetric ellipsoids defined by \eqref{eq:quasiellipsoid}. The minimum distance between any two scatterers is at least $0.5$, and all scatterers satisfy the same boundary condition. Figure~\ref{fig: farfield8} shows the geometry and representative far-field patterns for $k=10$. The scattered field is represented by the double-layer potential, and Table~\ref{tab:geo1doubleexterior8} reports the accuracy and timing results for $k=2,5,10$. The method achieves errors ranging from $10^{-7}$ to $10^{-13}$ over the tested wavenumbers and discretization levels, demonstrating robustness for multi-body wave interactions. As expected, higher wavenumbers require finer discretizations to achieve comparable accuracy. The timing results show that both the matrix assembly time $T_{\mathrm{mat}}$ and the linear solve time $T_{\mathrm{solve}}$ increase substantially with the total number of unknowns. For the larger discretizations in this eight-body example, $T_{\mathrm{solve}}$ becomes the dominant cost, reflecting the increasing expense of solving the coupled multiple scattering linear system.


\subsubsection{Example 4: Regular cubic array with block reuse}
\label{sec:cubic}


\begin{figure}[!t]
	\begin{minipage}[t]{0.45\textwidth}
		\centering
		\includegraphics[width = 2.7in]{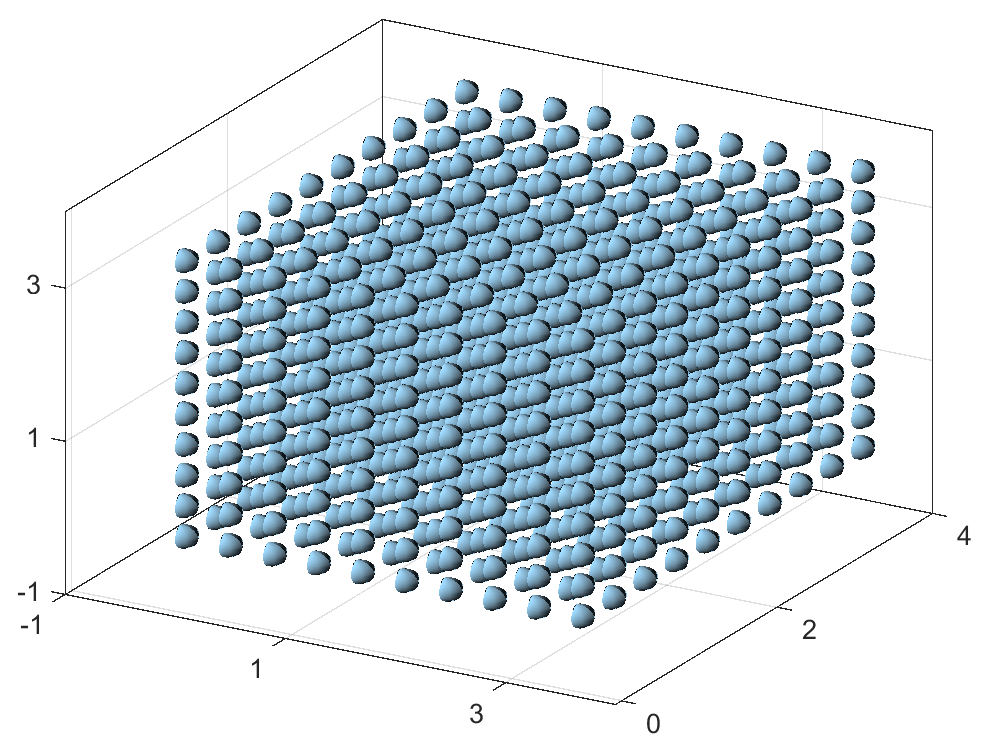}
\subcaption{One thousand quasi-axisymmetric ellipsoids.}
		\label{fig: blue 1000}
	\end{minipage}
	\begin{minipage}[t]{0.45\textwidth}
		\centering
		\includegraphics[width = 2.7in]{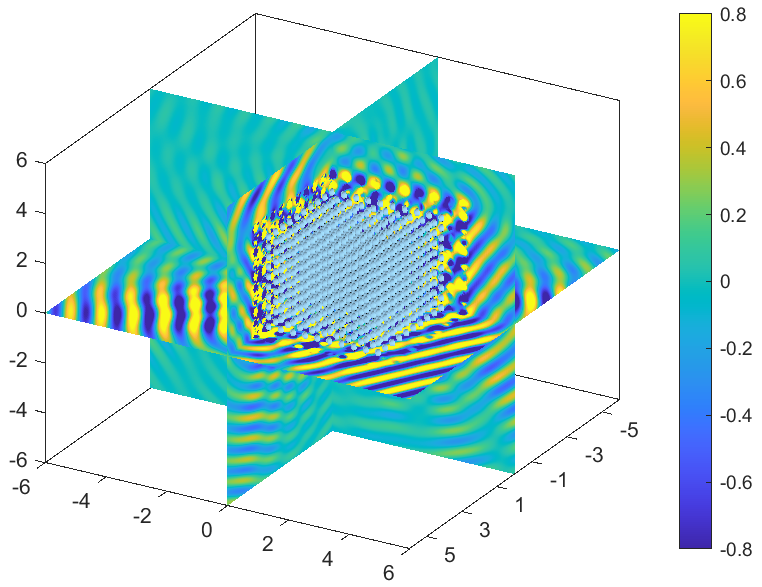}
		\subcaption{The real part of $u^{sc}$.}
		\label{fig:5}
	\end{minipage}
	\begin{minipage}[t]{0.45\textwidth}
		\centering
		\includegraphics[width=2in]{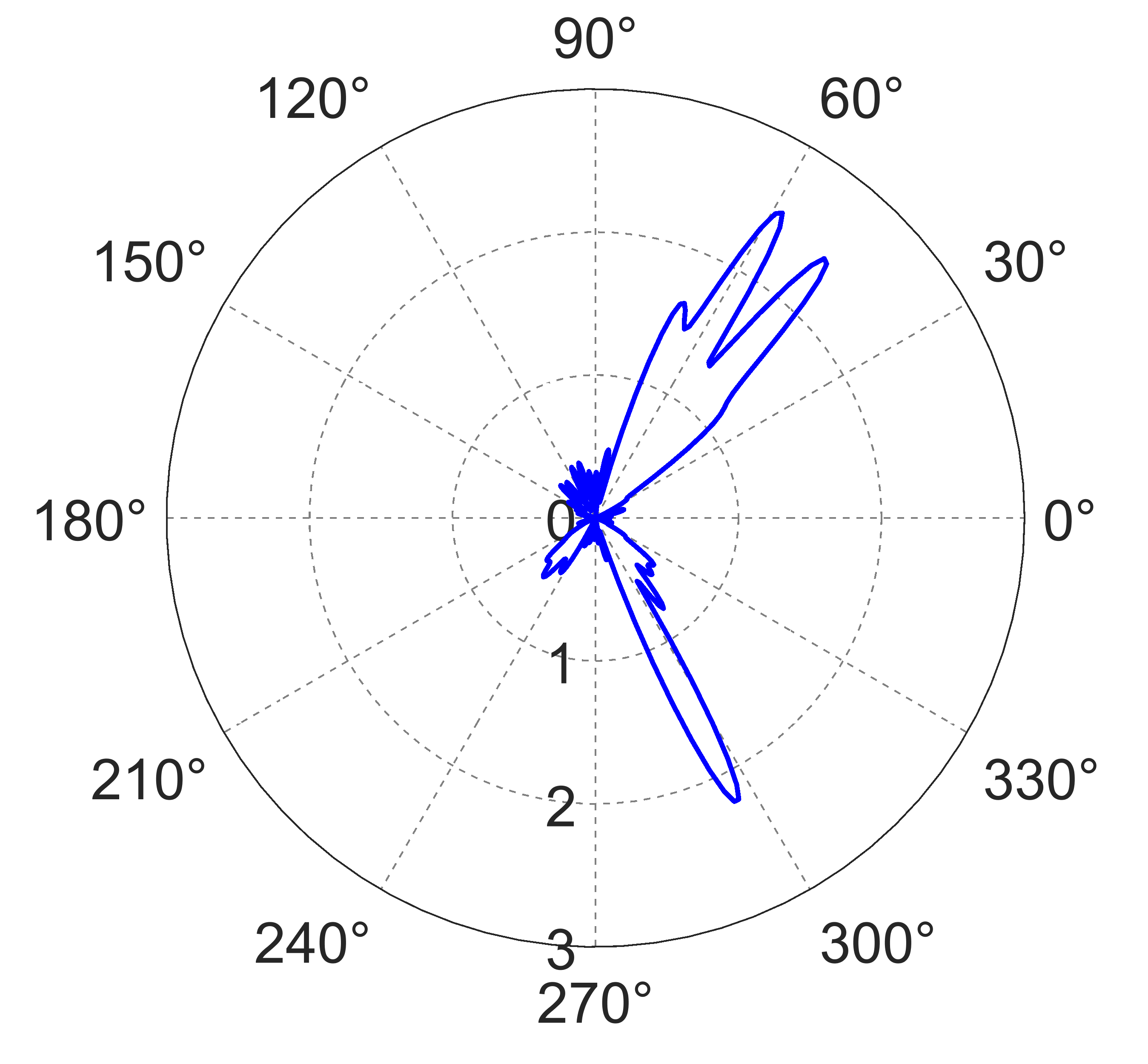}
		\subcaption{$|u_\infty(\cdot,\pi/2)|.$}
		\label{subfig: 1000 far theta}
	\end{minipage}
	\begin{minipage}[t]{0.45\textwidth}
		\centering
		\includegraphics[width=2in]{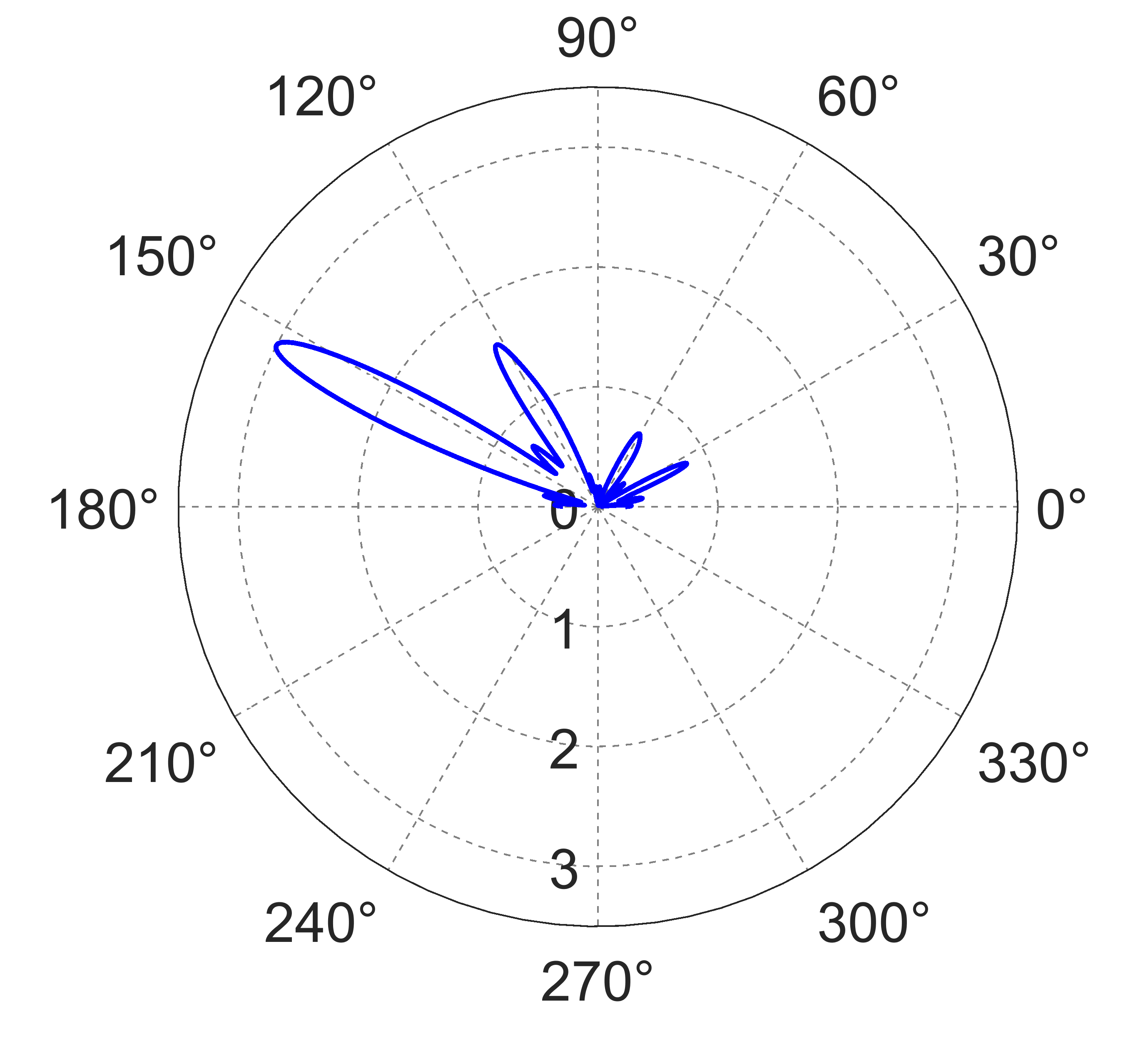}
		\subcaption{$|u_\infty(0,\cdot)|.$}
		\label{subfig: 1000 far phi}
	\end{minipage}
\caption{Regular cubic array with 1000 quasi-axisymmetric ellipsoids in the box $[-1,4]\times[-1,4]\times[-1,4]$. The minimum surface-to-surface distance is greater than $0.05$.}
		\label{fig: farfield1000}
\end{figure}

\begin{table*}[!t]
    \caption{Accuracy and efficiency for exterior scattering by regular cubic arrays using the double-layer potential. The discretization parameters are $N_p=1$, $N_f=6$, and $N_{\mathrm{pts}}=96$. The column ``Reuse'' gives the block reuse ratio
    $\left(1 - \frac{(2P-1)^3-1}{N_{\mathrm{geo}}(N_{\mathrm{geo}}-1)}\right)\times 100\%$.}
    \label{tab:multibody_cubic}
    \centering
    \renewcommand{\arraystretch}{0.8}
    \begin{tabular}
    {p{0.5cm}<{\raggedright}p{1cm}<{\raggedright}p{1.2cm}<{\raggedright}
    p{1.5cm}<{\raggedright}p{1.5cm}<{\raggedright}p{1.5cm}
    <{\raggedright}p{2cm}<{\raggedright}p{1.2cm}<{\raggedright}}
    \specialrule{0.1em}{4pt}{4pt}
    $k$ & $N_{\mathrm{geo}}$ & $N_{\mathrm{tot}}$ & $T_{\mathrm{ker}}$ & $T_{\mathrm{mat}}$ & $T_{\mathrm{solve}}$ & $\mathrm{Error}$ & Reuse \\
    \specialrule{0.05em}{4pt}{4pt}
    2  &   8 &   768 & 0.55 & 0.62  & 0.04   & 4.13E-07 & 53.6\% \\
    2  &  64 &  6144 & 0.56 & 1.52  & 0.68   & 6.28E-08 & 91.5\% \\
    2  & 216 & 20736 & 0.55 & 11.17 & 7.06   & 1.85E-08 & 97.1\% \\
    2  & 512 & 49152 & 0.49 & 45.78 & 35.81  & 7.34E-09 & 98.7\% \\
    2  & 1000 & 96000 & 0.61  & 206.38 &  188.76 & 3.37E-09  & 99.3\%\\
    \specialrule{0em}{3pt}{3pt}
    5  &   8 &   768 & 0.58 & 0.65  & 0.04   & 3.48E-07 & 53.6\% \\
    5  &  64 &  6144 & 0.61 & 1.60  & 0.96   & 4.37E-08 & 91.5\% \\
    5  & 216 & 20736 & 0.63 & 10.36 & 14.42  & 1.01E-08 & 97.1\% \\
    5  & 512 & 49152 & 0.56 & 47.90 & 93.27  & 2.80E-09 & 98.7\% \\
    5  & 1000 & 96000 & 0.57 & 220.16 & 307.51  & 2.09E-09  & 99.3\%\\
    \specialrule{0em}{3pt}{3pt}
    10 &   8 &   768 & 0.57 & 0.62  & 0.02   & 5.11E-07 & 53.6\% \\
    10 &  64 &  6144 & 0.59 & 1.58  & 1.47   & 8.39E-08 & 91.5\% \\
    10 & 216 & 20736 & 0.49 & 9.59  & 28.57  & 6.35E-08 & 97.1\% \\
    10 & 512 & 49152 & 0.56 & 48.37 & 282.76 & 4.66E-08 & 98.7\% \\
    10  & 1000 & 96000 & 0.58 & 242.67  & 2235.26 &  3.50E-08  & 99.3\%\\
    \specialrule{0.1em}{2pt}{0pt}
    \end{tabular}
\end{table*}

We next consider regular $P\times Q\times R$ cubic arrays of identical quasi-axisymmetric ellipsoids. Since all scatterers have the same orientation and lie on a lattice, many inter-body interaction blocks differ only by translation. The matrix assembly can therefore reuse previously computed blocks instead of constructing every pairwise interaction independently.

To accelerate the computation, the linear solver is chosen according to the wavenumber. For $k=2$ and $k=5$, we use GMRES with a block-diagonal preconditioner formed from the LU factorization of the single-body self-interaction matrix. This factorization is computed once and shared by all scatterers. For $k=10$, stronger inter-body coupling slows GMRES convergence, so we instead use the direct LU solver. This hybrid strategy balances efficiency at lower frequencies with robustness at the higher frequency tested here.

Figure~\ref{fig: farfield1000} shows the $10\times10\times10$ array with 1000 scatterers and the corresponding scattered field for $k=2$. The scatterers lie in the box $[-1,4]\times[-1,4]\times[-1,4]$, with a minimum surface-to-surface distance greater than $0.05$. The real part of $u^{sc}$ on three cross-sectional planes displays the collective wave interactions in the array, while the far-field cuts $|u_\infty(\cdot,\pi/2)|$ and $|u_\infty(0,\cdot)|$ show the multiple-lobe structure typical of periodic multi-body scattering.

Table~\ref{tab:multibody_cubic} reports the results for the exterior double-layer formulation with $N_p=1$, $N_f=6$, $N_{\mathrm{pts}}=96$, and wavenumbers $k=2,5,10$. As the array size increases from $2^3=8$ to $10^3=1000$ scatterers, the block reuse ratio rises from $53.6\%$ to $99.3\%$. Consequently, the kernel evaluation time $T_{\mathrm{ker}}$ remains nearly constant, since only a fixed set of distinct translated interaction blocks must be evaluated. The matrix construction time $T_{\mathrm{mat}}$ grows with the array size, but much more moderately than it would without block reuse. The absolute error remains between $10^{-7}$ and $10^{-9}$, which is consistent with the deliberately coarse discretization used in this large-scale experiment.

\subsubsection{Example 5: Robustness under random rotations and position perturbations}
\label{sec:random}


\begin{figure}[!t]
    \begin{minipage}[t]{0.45\textwidth}
        \centering
        \includegraphics[width=2.7in]{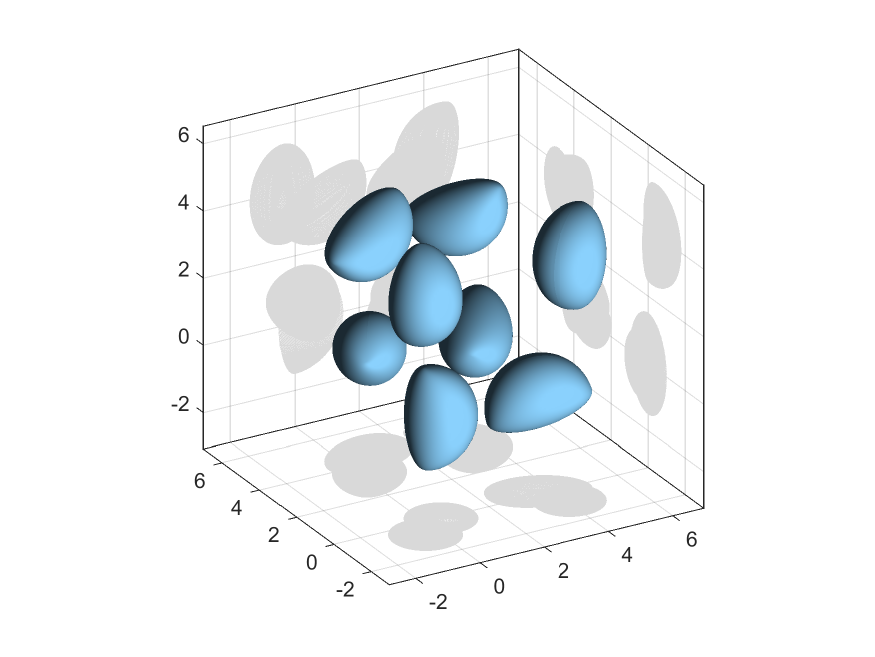}
        \subcaption{$2\times2\times2$ array.}
        \label{fig:randomarray2}
    \end{minipage}
    \begin{minipage}[t]{0.45\textwidth}
        \centering
        \includegraphics[width=2.7in]{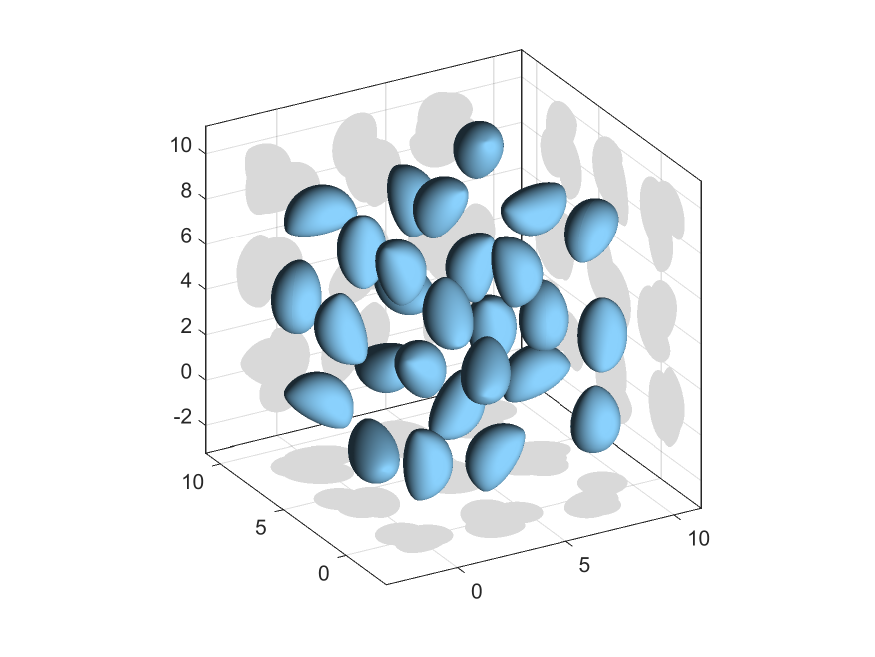}
        \subcaption{$3\times3\times3$ array.}
        
    \end{minipage}
    \caption{Illustration of a randomly configured multi-body scattering problem.
    Each scatterer is independently rotated by random angles $\theta_1\sim\mathcal{U}[0,2\pi)$
    and $\theta_2\sim\mathcal{U}[0,2\pi)$, and displaced by a random perturbation
    $\boldsymbol{\varepsilon}\sim\mathcal{U}[-\varepsilon_{\max},\varepsilon_{\max}]^3$.}
    \label{fig:randomarray}
\end{figure}

\begin{table*}[!t]
    \caption{Accuracy and efficiency for exterior scattering by randomly configured arrays with random rotations and position perturbations.}
    \label{tab:random_array}
    \centering
    \renewcommand{\arraystretch}{0.8}
    \begin{tabular}
    {p{0.5cm}<{\raggedright}p{0.8cm}<{\raggedright}p{0.6cm}<{\raggedright}p{0.6cm}<{\raggedright}
    p{0.8cm}<{\raggedright}p{1cm}<{\raggedright}p{1.3cm}<{\raggedright}p{1.3cm}<{\raggedright}
    p{1.3cm}<{\raggedright}p{2cm}<{\raggedright}}
    \specialrule{0.1em}{4pt}{4pt}
    $k$ & $N_{\mathrm{geo}}$ & $N_p$ & $N_f$ & $N_{\mathrm{pts}}$ & $N_{\mathrm{tot}}$ & $T_{\mathrm{ker}}$ & $T_{\mathrm{mat}}$ & $T_{\mathrm{solve}}$ & $\mathrm{Error}$ \\
    \specialrule{0.05em}{4pt}{4pt}
    2 &  8 & 6 & 40 & 3840 & 30720 & 145.37 & 287.00 & 98.56  & 7.70E-11 \\
    2 & 27 & 4 & 40 & 2560 & 69120 & 115.93 & 739.71 & 911.58 & 5.56E-10 \\
    \specialrule{0em}{3pt}{3pt}
    5 &  8 & 8 & 40 & 5120 & 40960 & 267.17 & 590.39 & 305.19 & 1.60E-12 \\
    5 & 27 & 4 & 40 & 2560 & 69120 & 97.48  & 751.03 & 680.02 & 6.61E-12 \\
    \specialrule{0em}{3pt}{3pt}
    10 &  8 & 6 & 60 & 5760 & 46080 & 272.89 & 521.36 & 278.43 & 9.20E-11 \\
    10 & 27 & 6 & 40 & 3840 & 103680 & 159.45  & 1697.75 & 2077.57 & 3.81E-11 \\
    \specialrule{0.1em}{2pt}{0pt}
    \end{tabular}
\end{table*}

We next test the robustness of the solver for arrays in which each scatterer is independently rotated and displaced from a regular lattice. Starting from a reference quasi-axisymmetric ellipsoid, the $i$th scatterer is assigned the random rotation
$R_i = R_z(\theta_1^{(i)})R_y(\theta_2^{(i)})$, where
$\theta_1^{(i)}, \theta_2^{(i)} \sim \mathcal{U}[0, 2\pi)$ are independent uniform random variables and
\begin{equation}
    R_z(\theta_1) =
    \begin{pmatrix}
        \cos\theta_1 & -\sin\theta_1 & 0 \\
        \sin\theta_1 &  \cos\theta_1 & 0 \\
        0 & 0 & 1
    \end{pmatrix}, \quad
    R_y(\theta_2) =
    \begin{pmatrix}
        \cos\theta_2 & 0 & \sin\theta_2 \\
        0 & 1 & 0 \\
        -\sin\theta_2 & 0 & \cos\theta_2
    \end{pmatrix}.
\end{equation}
The surface points and outward normals of the $i$th scatterer are obtained from the reference geometry by
\begin{equation}
    \mathbf{x}^{(i)} = R_i\mathbf{x}^{(0)} + \mathbf{c}_i, \quad
    \mathbf{n}^{(i)} = R_i\mathbf{n}^{(0)},
\end{equation}
where $\mathbf{x}^{(0)}$ and $\mathbf{n}^{(0)}$ are the reference surface points and normals.
The center of the $(p,q,r)$th scatterer in a $P\times Q\times R$ array is placed at
\begin{equation}
    \mathbf{c}_i = (p\delta,\, r\delta,\, q\delta)^T + \boldsymbol{\epsilon}_i,
    \quad \boldsymbol{\epsilon}_i \sim
    \mathcal{U}[-\epsilon_{\max}, \epsilon_{\max}]^3,
\end{equation}
where $\delta=8a$ is the nominal spacing, $a$ is the scaling factor of the reference geometry, and $\epsilon_{\max}=0.1\delta$. Let $\rho_{\rm ref}$ denote the radius of the minimum enclosing ball of the reference scatterer. Then the minimum surface-to-surface distance is bounded below by
\[
d_{\min}
\ge
\delta - 2a\rho_{\rm ref} - 2\epsilon_{\max}>0,
\]
which guarantees well-separation of all scatterers.

Figure~\ref{fig:randomarray} illustrates the resulting $2\times2\times2$ and $3\times3\times3$ configurations. Table~\ref{tab:random_array} reports the corresponding exterior scattering results for $N_{\mathrm{geo}}=8$ and $N_{\mathrm{geo}}=27$ scatterers. The errors remain small for all tested wavenumbers and array sizes. Specifically, they are of order $10^{-11}$--$10^{-10}$ for $k=2$, improve to about $10^{-12}$ for $k=5$, and remain at the $10^{-11}$ level for $k=10$. Thus, random orientations and nonuniform positions do not cause a noticeable loss of accuracy, provided that the scatterers remain well separated.

The timings also reflect the additional cost of general random configurations. Unlike the regular arrays considered above, these configurations do not admit block reuse, since the relative orientations and positions of the scatterers are all different. Consequently, the matrix construction time $T_{\mathrm{mat}}$ grows rapidly with $N_{\mathrm{geo}}$. For example, when $N_{\mathrm{geo}}=27$ and $k=10$, $T_{\mathrm{mat}}$ reaches $1697.75$s and the solve time $T_{\mathrm{solve}}$ reaches $2077.57$s. These results confirm that the solver remains robust for arbitrarily oriented, nonuniformly placed scatterers.

\section{Conclusion}

This paper presents an efficient, high-accuracy boundary integral method for solving multiple scattering problems on quasi-axisymmetric surfaces in $\mathbb{R}^3$. By exploiting the azimuthal periodicity of these geometries, the proposed solver combines FFT-based discrete convolution, kernel splitting, and recurrence relations for the efficient evaluation of modal Green's functions and their derivatives. Coupled with generalized Gaussian quadrature and a high-order Nyström discretization, this framework delivers accurate and efficient simulation of wave scattering for a broad class of quasi-axisymmetric structures.

The present work provides a robust foundation for several natural extensions. Future work includes the incorporation of the fast multipole method and hierarchical matrix compression to improve scalability for much larger configurations. Another direction is the optimization of systems composed of multiple acoustic scatterers, such as acoustic metasurfaces and engineered scattering arrays \cite{article44}. Because the present solver combines high accuracy with the ability to treat many interacting bodies, it provides a natural forward model for shape optimization problems in which the locations, orientations, and geometries of the scatterers must be tuned to achieve prescribed wave responses. Other important directions include extending the framework to layered media scattering \cite{article53} and inverse scattering problems.

\begin{mycomment}
We can approximate the scattering field at a specified position through numerical solutions. The main research results are as follows
\begin{itemize}
    \item We propose parameter equations for quasi-axisymmetric target bodies and summarize a general calculation format for equations with similar integral formats based on the high-precision FFT numerical algorithm for axisymmetric objects.
    \item We expand the periodic function into node values at discrete points by introducing the Lagrange basis polynomial and construct the modal form of the Green's function. Then using the principle of convolution and the Fourier transform we can calculate the modal function corresponding to the Helmholtz equation at a given wavenumber by reconstructing the relative position relationship between the target point and the source point. Using single-layer potential and double-layer potential respectively to represent the solution of the equation, we can construct a scalar linear system with scale determined by the number of panels, order, and azimuth partition density. The numerical algorithm can achieve accuracy of $10^{-10}$ or above for quasi-axisymmetric objects with smooth surfaces
    \item We extend the numerical algorithm for single scattering problem to multiple scattering problem. Specifically, we utilize high-order Gaussian quadrature to discretize each surface. We consider the interaction between target bodies and the self-interaction of each target body. For the same object, we avoid duplicate construction of matrix elements and reduce computational costs by translation. 
\end{itemize}
\end{mycomment}


\end{document}